\documentclass[a4paper,12pt,onecolumn]{article}
\usepackage[spanish]{babel}
\usepackage[latin1]{inputenc}
\usepackage[vmargin=2cm,hmargin=2cm,headheight=14.5pt,top=2cm,headsep=.5cm]{geometry}
\usepackage{mathptmx}
\usepackage{bm}
\usepackage{empheq}
\usepackage{stackrel}
\usepackage{cases}
\usepackage{mathtools}
\usepackage{amsthm,amsmath,amscd}
\usepackage{amssymb,amsfonts}
\usepackage{makeidx}
\usepackage[all]{xy}
\usepackage{mathptmx}
\usepackage{perpage}
\usepackage[svgnames]{xcolor}
\usepackage[symbol]{footmisc}

\MakePerPage[2]{footnote}
\usepackage{graphicx}
\DeclareMathSizes{12}{12}{8}{6}

\usepackage{cite}
\usepackage{url}
\usepackage{listings}
\newtheoremstyle{ptheorem}{1em}{0em}{\itshape}{}{\bfseries}{.}{.5em}{}
\theoremstyle{ptheorem}
\newtheorem{thm}{Teorema}
\newtheorem{pro}[thm]{Proposición}
\newtheorem{lem}[thm]{Lema}
\newtheorem{cor}[thm]{Corolario}

\theoremstyle{definition}
\newtheorem{dfn}{Definición}

\theoremstyle{remark}

\newtheorem{rem}{Observación}

\DeclareMathOperator{\sign}{sign}

\DeclareMathOperator{\dif}{d}

\DeclareMathOperator{\arctanh}{arctanh}

\newcommand{\cC}{{\mathcal C}}

\newcommand{\cF}{{\mathcal F}}

\newcommand{\cH}{{\mathcal H}}

\newcommand{\bC}{{\mathbb C}}

\newcommand{\bR}{{\mathbb R}}

\newcommand{\bZ}{{\mathbb Z}}

\renewcommand{\a}{\alpha}
\renewcommand{\b}{\beta}
\renewcommand{\c}{\gamma}
\renewcommand{\l}{\lambda}
\newcommand{\e}{\epsilon}

\renewcommand{\phi}{\varphi}

\newcommand{\ol}{\overline}
\newcommand{\fa}{\forall}

\newcommand{\nkp}{\enskip}

\newcommand{\sfa}{\nkp\fa}
\renewcommand{\d}{\delta}

\renewcommand{\(}{\left(}
\renewcommand{\)}{\right)}
\renewcommand{\[}{\left[}
\renewcommand{\]}{\right]}

\newcommand{\til}{\tilde}

\newcommand\ctp{en casi todas partes }
\newcommand\ct{para casi todo }

\newcommand{\subt}[1]{\vspace*{5mm}
\noindent\textbf{#1}\vspace*{3mm}
}

\newcommand{\ssubt}[1]{\vspace*{3mm}
\noindent\textit{\textbf{#1}}\vspace*{2mm}
}

\begin{document}

\title{Funciones de Green para Ecuaciones Diferenciales con Involuciones\thanks{Financiado por FEDER y Ministerio de Economía y Competitividad, España, proyecto MTM2010-15314.}}
\author{Fernando Adrián Fernández Tojo\footnote{Financiado por una beca FPU, Ministerio de Educaci\'on, Cultura y Deporte, España.}\\\small \textit{Departamento de An\'alise Ma\-te\-m\'a\-ti\-ca, Facultade de Matem\'aticas,}\\ \small
\textit{Universidade de Santiago de Com\-pos\-te\-la, España.}\\\small
\textit{Correspondencia: fernandoadrian.fernandez@usc.es}}
\date{2014}

\maketitle

\subt{Resumen}\\
En este trabajo se revisan los más recientes avances relativos a problemas con ecuaciones diferenciales que tratan involuciones. Tenemos en cuenta dos casos: el caso con condiciones iniciales y coeficientes constantes y el caso de condiciones de contorno homogéneas y coeficientes variables. Usando diferentes técnicas, obtenemos la función de Green para ambos problemas, estudiamos su signo y derivamos resultados de él.
\vspace*{5mm}\\
\noindent \textbf{Palabras clave:} Involución, reflexión, función de Green, principio del máximo, ecuaciones diferenciales.

\subt{Abstract}\\
In this work we revise the most recent developments concerning the study of first order problems regarding differential equations with involutions. We take into account two cases: the case of initial conditions and constant coefficients and the case of homogeneous boundary conditions and non-constant coefficients. Using different techniques, we obtain the Green's function for both problems, study its sign and derive further results.


\subt{Introducción}\\
El estudio de ecuaciones funcionales con involuciones (EDI) comienza con la solución de la ecuación $x'(t)=x(1/t)$ llevada a cabo por Silberstein (ver \cite{Sil}) en 1940. En pocas palabras, una involución es simplemente una función $f$ tal que, sin ser la identidad, satisface $f(f(x))=x$ para todo $x$ en su dominio de definición.
Para la mayoría de las aplicaciones en análisis, la involución está definida en un intervalo de $\bR$ y es continua, lo cual implica que es decreciente y tiene un único punto fijo (ver \cite{Wie}). Desde el artículo fundacional de Silberstein, el estudio de problemas con EDI ha estado principalmente basado en los casos con condiciones iniciales y en especial en el caso de la reflexión $f(x)=-x$.
Por citar algunos ejemplos, Wiener y Watkins estudian en \cite{Wie} la solución de la ecuación $x'(t)-a\, x(-t)=0$ con condiciones iniciales. La ecuación $x'(t)+a\, x(t)+b\,x(-t)=g(t)$ ha sido tratada en \cite{Pia, Pia2}. En \cite{Kul, Sha, Wie, Wat1, Wie2} se introducen algunos resultados que permiten transformar este tipo de problemas con involuciones y condiciones iniciales en ecuaciones diferenciales ordinarias (EDO) de segundo orden y condiciones iniciales o bien sistemas de EDO de primer orden y dimensión dos, garantizando además que la solución del segundo problema será una solución del problema original. Además, las propiedades asintóticas y de acotación de las soluciones de problemas iniciales de primer orden se estudian en \cite{Wat2} y \cite{Aft} respectivamente. Se consideran problemas de contorno de segundo orden en \cite{Gup, Gup2, Ore2, Wie2} para condiciones de contorno de Dirichlet y de Sturm-Liouville y en \cite{Ore} se estudian ecuaciones de orden superior. Otras técnicas aplicadas a problemas con reflexión se pueden encontrar en \cite{And, Ma, Wie1}.\par
Por otra parte, en \cite{Cab4, Cab5, Cab6} se estudia el caso particular de una ecuación de primer orden con condiciones de contorno periódicas y generales, añadiendo un nuevo elemento a los estudios previos: la función de Green. En dichos trabajos se estudia el signo de la función de Green y se deducen principios del máximo y del antimáximo.\par
Este trabajo, siguiendo \cite{Cab7, Cab8}, muestra el desarrollo de dichas ecuaciones y extiende los resultados de la bibliografía antes mencionada con la obtención de funciones de Green en casos más generales, como los que citaremos posteriormente. Intentamos por tanto responder a las siguientes cuestiones: ¿De qué maneras se pueden encontrar soluciones para un problema con involuciones?, ¿Cómo se obtiene la función de Green de un problema lineal de orden uno con una involución?, ¿Existe relación entre problemas con distintas involuciones? Para enfrentarnos a estas preguntas estudiaremos dos partes claramente diferenciadas: el caso del problema de valor inicial \cite{Cab8} y coeficientes constantes y el caso del problema con condiciones de contorno homogéneas y coeficientes variables \cite{Cab7}.\par 
En cada caso se utilizará un método diferente para la obtención de las distintas funciones de Green. Dado que cada uno de estos métodos es fácilmente extrapolable al otro caso, lo aquí expuesto dará pie a la realización de progresos posteriores en el área.
La distribución de este trabajo será como sigue. en la Sección 1 nos dedicaremos al estudio del problema de valor inicial para, en la Sección 3, aplicar los resultados obtenidos a ejemplos concretos que nos permitan ilustrar la teoría. A partir de la Sección 4 analizamos el problema con condiciones de contorno homogéneas, haciendo hincapié en el Teorema de Equivalencia de Involuciones en la Sección 5 mientras que en la Sección 6 estudiamos la ecuación para después analizar, en las Secciones 7, 8 y 9, los distintos casos que se pueden presentar, terminando con aquellos en los que no se puede derivar la función de Green de forma explícita.

\subt{Soluciones del problema de valor inicial}\\
Siguiendo \cite{Cab8}, consideremos $h\in L^1(\bR)$, $t_0$, $a$, $b\in\bR$. Estudiamos aquí el siguiente problema
\begin{equation}\label{gpabconst} x'(t)+a\,x(-t)+b\,x(t)=h(t),\quad x(t_0)=c.\end{equation}
Para hacerlo, primero estudiamos la ecuación homogénea
\begin{equation}\label{heabconst} x'(t)+a\,x(-t)+b\,x(t)=0.\end{equation}
Derivando y haciendo las sustituciones adecuadas llegamos a la ecuación
\begin{equation}\label{rheabconst} x''(t)+(a^2-b^2)x(t)=0.\end{equation}
Sea $\omega:=\sqrt{|a^2-b^2|}$. La ecuación \eqref{rheabconst} presenta tres casos distintos:\par
\textbf{(C1). $a^2>b^2$.} En este caso, $u(t)=\a\cos \omega t+\b\sen\omega t$ es una solución de \eqref{rheabconst} para cada $\a,\,\b\in\bR$. Si imponemos la ecuación \eqref{heabconst} a esta expresión llegamos a la solución general
$$u(t)=\a(\cos\omega t-\frac{a+b}{\omega}\sen \omega t)$$
de la ecuación \eqref{heabconst} con $\a\in\bR$.\par 
\textbf{(C2). $a^2<b^2$.} Ahora, $u(t)=\a\cosh \omega t+\b\senh\omega t$ es una solución del problema \eqref{rheabconst} para cada $\a,\,\b\in\bR$. Si imponemos la condición \eqref{heabconst} a esta fórmula llegamos a la solución general
$$u(t)=\a(\cosh\omega t-\frac{a+b}{\omega}\senh \omega t)$$
de la ecuación \eqref{heabconst} con $\a\in\bR$.\par 
\textbf{(C3). $a^2=b^2$.} En este caso, $u(t)=\a t+\b$ es una solución de  \eqref{rheabconst} para cada $\a,\,\b\in\bR$. Si imponemos la ecuación \eqref{heabconst} obtenemos dos casos.\par
\textbf{(C3.1). $a=b$,} donde
$$u(t)=\a(1-2\,a\ t)$$
es la solución general de la ecuación \eqref{heabconst} con $\a\in\bR$, y\par 
\textbf{(C3.2). $a=-b$,} donde
$$u(t)=\a$$
es la solución general de la ecuación \eqref{heabconst} con $\a\in\bR$.\par 

Ahora, consideremos $\til u$ y $\til v$ tales que satisfacen
\begin{align*}\til u'(t)+a\til u(-t)+b\til u(t) & =0, \quad \til u(0)=1,\\
\til v'(t)-a\til v(-t)+b\til v(t) & =0,\quad \til v(0)=1.\end{align*}
Obsérvese que $\til u$ y $\til v$ se pueden obtener de las expresiones explícitas de los casos (C1)--(C3) simplemente tomando $\a=1$. También cabe destacar que, si $\til u$ está en el caso (C3.1), $\til v$ esta en el caso (C3.2) y viceversa, mientras que si $\til u$ esta en el caso (C1) (respectivamente (C2)) $\til v$ también y viceversa.
\par
Consideremos ahora las partes pares e impares de la función $f$, esto es, $f_e(x):=[f(x)+f(-x)]/2$ y $f_o(x):=[f(x)-f(-x)]/2$, como se hace en \cite{Cab4}. Tenemos ahora las siguientes propiedades de las funciones $\til u$ y $\til v$.
\begin{lem} \label{auxlemuv}Para cada $t,s\in\bR$, las siguientes propiedades se satisfacen.
\begin{enumerate}
\item $\til u_e\equiv\til v_e$, $\til u_o\equiv k\,\til v_o$ para alguna constante $k$ en c.t.p.,
\item$u_e(s)v_e(t)=u_e(t)v_e(s)$, $u_o(s)v_o(t)=u_o(t)v_o(s)$,
\item $\til u_e\til v_e-\til u_o\til v_o\equiv 1,$
\item $\til u(s)\til v(-s)+\til u(-s)\til v(s)=2[\til u_e(s)\til v_e(s)-\til u_o(s)\til v_o(s)]=2.$
\end{enumerate}
\end{lem}
\begin{proof}
$(I)$ y $(III)$ se pueden comprobar inspeccionando los diferentes casos. $(II)$ es una consecuencia directa de $(I)$. $(IV)$ se obtiene de las definiciones de parte par e impar y de $(III)$.
\end{proof}
El siguiente Teorema usará las funciones $\til u$ y $\til v$ y sus propiedades para encontrar las soluciones del problema \eqref{gpabconst}.
\begin{thm}\label{thmconstsol} El problema  \eqref{gpabconst} tiene una única solución maximal definida en $\bR$ si y solo si  $\til u(t_0)\ne 0$.
\end{thm}
\begin{proof}Primero obtendremos una solución $u$ de la ecuación $x'(t)+a\,x(-t)+b\,x(t)=h(t)$.\par
Definamos
$$\phi:=h_o\til v_e-h_e\til v_o,\quad\text{y}\quad \psi:h_e\til u_e-h_o\til u_o.$$
Obsérvese que $\phi$ es impar, $\psi$ es par y $h=\phi \til u+\psi\til v$ por el Lema \ref{auxlemuv} $(III)$. Puesto que estamos trabajando con el operador lineal $(Lx)(t)=x'(t)+a\,x(-t)+b\,x(t)$, basta encontrar $y$ y $z$ tales que $Ly=\phi\til u$ y $Lz=\psi \til v$ dado que, en dicho caso, podemos definir $u=y+z$, y concluir que $Lu=h$.\par
Para resolver $Ly=\phi \til u$ usaremos el método de variación de constantes. Tómese $y=\til\phi\,\til u$, y asumamos que $\til\phi$ es par (enseguida se verá que esto puede hacerse). Entonces,
$$Ly(t)=\til\phi'(t)\til u(t)+\til\phi(t)[\til u'(t)+a\,\til u(-t)+b\,\til u(t)]=\til\phi'(t)\til u(t).$$
Por lo tanto, tomando $\til\phi(t):=\int_{k_1}^t\phi(s)\dif s$ (par) para algún $k_1\in\bR$, la función $y$ satisface $Ly=\phi \til u$.
Del mismo modo, tomando $z=\til\psi\til v$ con $\til\psi(t):=\int_{k_1}^t\psi(s)\dif s$ (obsérvese que $\til\psi$ es impar), tenemos que $Lz=\psi \til v$.
\par
Ahora, sea $u_{0}:=y+z=\til\phi\,\til u+\til\psi\til v$ con $k_1=k_2=0$. Obsérvese que $u_{0}$ satisface $L\,u_{0}=h$ y $u_{0}(0)=0$. Definiendo $w=u_0+\frac{c-u_0(t_0)}{\til u(t_0)}\til u$, $w$ es una solución del problema \eqref{gpabconst}. Veamos que es única.\par
Supondremos ahora que $w_1$ y $w_2$ son soluciones de \eqref{gpabconst}. Entonces $w_2-w_1$ es una solución de \eqref{heabconst}. Por lo tanto, $w_2-w_1$ es de una de las formas cubiertas por los casos (C1)--(C3) y, en cualquier caso, un múltiplo de $\til u$, esto es, $w_2-w_1=\l\,\til u$ para algún $\l\in\bR$. También resulta evidente que $(w_2-w_1)(t_0)=0$, pero tenemos que $\til u(t_0)\neq 0$ por hipótesis, por lo tanto $\l=0$ y $w_1=w_2$. Esto es, el problema  \eqref{gpabconst} tiene solución única.\par
Supongamos ahora que $w$ es una solución de \eqref{gpabconst} y $\til u(t_0)=0$. Entonces $w+\l\,\til u$ también es una solución de \eqref{gpabconst} para cada $\l\in\bR$, lo cual demuestra el resultado.
\end{proof}
Este último Teorema plantea una cuestión evidente: ¿Bajo qué circunstancias $\til u(t_0)\neq0$? Para responder a esta pregunta es suficiente con estudiar los casos (C1)--(C3).
\begin{lem}
 $\til u(t_0)=0$ solo en los siguientes casos,
 \begin{itemize}
\item si $a^2>b^2$ y $t_0=\frac{1}{\omega}\arctan\frac{\omega}{a+b}+k\pi$ para algún $k\in\bZ$,
\item si $a^2<b^2$, $a\,b>0$ y $t_0=\frac{1}{\omega}\arctanh \frac{\omega}{a+b}$,
\item si $a=b$ y $t_0=\frac{1}{2a}$.
 \end{itemize}
\end{lem}
\begin{dfn} Sea $t_1,t_2\in\bR$. Definimos la \textbf{función característica orientada} para el par $(t_1,t_2)$ como
$$\chi_{t_1}^{t_2}(t):=\begin{cases}1, & t_1\le t\le t_2\\-1, & t_2\le t< t_1\\0, & \text{en otro caso.}\end{cases}$$
\end{dfn} 
\begin{rem}La anterior definición implica que, para cada función integrable $f:\bR\to\bR$,
$$\int_{t_1}^{t_2}f(s)\dif s=\int_{-\infty}^{\infty}\chi_{t_1}^{t_2}(s)f(s)\dif s.$$
Además, $\chi_{t_1}^{t_2}=-\chi_{t_2}^{t_1}$.
\end{rem}
El siguiente Corolario proporciona la expresión de la función de Green para el problema \eqref{gpabconst}, de nuevo usando las funciones $\til u$ y $\til v$.
\begin{cor} Supongamos que $\til u(t_0)\ne0$. Entonces, la única solución para el problema \eqref{gpabconst} viene dada por
$$u(t):=\int_{-\infty}^\infty G(t,s)h(s)\dif s+\frac{c-u_0(t_0)}{\til u(t_0)}\til u(t),\quad t\in\bR,$$
donde
\begin{equation}\begin{aligned}\label{eqsolgen}G(t,s):= & \frac{1}{2}\([\til u(-s)\til v(t)+\til v(-s)\til u(t)]\chi_{0}^t(s)\right.\\ &\left.+[\til u(-s)\til v(t)-\til v(-s)\til u(t)]\chi_{-t}^{0}(s)\),\end{aligned}\quad t,s\in\bR.\end{equation}
\end{cor}
\begin{proof}
\begin{align}\notag &  u'(t)-\frac{c-u_0(t_0)}{\til u(t_0)}\til u'(t)  \\ \notag =  &\frac{1}{2}\(\frac{\dif}{\dif t}\int_{0}^t\[\til u(-s)\til v(t)+\til v(-s)\til u(t)\]h(s)\dif s\right. \\ \notag & \left. +\frac{\dif}{\dif t}\int_{-t}^{0}\[\til u(-s)\til v(t)-\til v(-s)\til u(t)\]h(s)\dif s\)\\ \notag  = &  \frac{1}{2}\(\frac{\dif}{\dif t}\int_{0}^t \[\til u(-s)\til v(t)+\til v(-s)\til u(t)\]h(s)\dif s\right. \\ \notag & \left.+\frac{\dif}{\dif t}\int_{0}^t\[\til u(s)\til v(t)-\til v(s)\til u(t)\]h(-s)\dif s\)\\ \notag = & h(t)+\frac{1}{2}\(\int_{0}^t \[\til u(-s)\til v'(t)+\til v(-s)\til u'(t)\]h(s)\dif s\right. \\ \label{equprima} & \left.+\int_{0}^t\[\til u(s)\til v'(t)-\til v(s)\til u'(t)\]h(-s)\dif s\).
\end{align}

Por otra parte,

\begin{align}\notag 
& a\[u(-t)-\frac{c-u_0(t_0)}{\til u(t_0)}\til u(-t)\]+b\[u(t)-\frac{c-u_0(t_0)}{\til u(t_0)}\til u(t)\] \\\notag  = & \frac{1}{2}a\int_{0}^{-t}\([\til u(-s)\til v(-t)+\til v(-s)\til u(-t)]h(s)\right. \\ \notag & \left.+[\til u(s)\til v(-t)-\til v(s)\til u(-t)]h(-s)\)\dif s \\\notag  & +  \frac{1}{2}b\int_{0}^{t}\([\til u(-s)\til v(t)+\til v(-s)\til u(t)]h(s)\right. \\ \notag & \left.+[\til u(s)\til v(t)-\til v(s)\til u(t)]h(-s)\)\dif s \\ \notag = & -\frac{1}{2}a\int_{0}^{t}\([\til u(s)\til v(-t)+\til v(s)\til u(-t)]h(-s)\right. \\ \notag & \left.+[\til u(-s)\til v(-t)-\til v(-s)\til u(-t)]h(s)\)\dif s \\\notag  & + \frac{1}{2} b\int_{0}^{t}\([\til u(-s)\til v(t)+\til v(-s)\til u(t)]h(s)\right. \\ \notag & \left.+[\til u(s)\til v(t)-\til v(s)\til u(t)]h(-s)\)\dif s \\\notag  = & \frac{1}{2} \int_{0}^{t}(-a[\til u(-s)\til v(-t)-\til v(-s)\til u(-t)] \\ \notag & +b[\til u(-s)\til v(t)+\til v(-s)\til u(t)])h(s)\dif s\\\notag  &+  \frac{1}{2} \int_{0}^{t}(-a[\til u(s)\til v(-t)+\til v(s)\til u(-t)]\\ \notag & +b[\til u(s)\til v(t)-\til v(s)\til u(t)])h(-s)\dif s\\\notag  = &
\frac{1}{2} \int_{0}^{t}(\til u(-s)[-a\til v(-t)+b\til v(t)]+\til v(-s)[a\til u(-t)+b\til u(t)]h(s)\dif s\\\notag  &+  \frac{1}{2} \int_{0}^{t}(\til u(s)[-a\til v(-t)+b\til v(t)]-\til v(s)[a\til u(-t)+b\til u(t)])h(-s)\dif s\\\notag  = &
-\frac{1}{2}\( \int_{0}^{t}(\til u(-s)v'(t)+\til v(-s)u'(t))h(s)\dif s \right. \\ \label{eqrest}  & \left.+ \int_{0}^{t}(\til u(s)v'(t)-\til v(s)u'(t))h(-s)\dif s\).
\end{align}
Por tanto, sumando \eqref{equprima} y \eqref{eqrest}, queda claro que $u'(t)+a\,u(-t)+b\,u(t)=h(t)$.\par
 Comprobamos ahora la condición inicial.
\begin{equation*}\begin{aligned} u(t_0)= & c-u_0(t_0)+\frac{1}{2}\int_{0}^{t_0}\([\til u(-s)\til v(t_0)+\til v(-s)\til u(t_0)]h(s)\right. \\ & \left.+[\til u(s)\til v(t_0)-\til v(s)\til u(t_0)]h(-s)\)\dif s.\end{aligned}\end{equation*}
Usando la construcción de la solución proporcionada en el Teorema \ref{thmconstsol}, resulta fácil probar ahora que
 \begin{equation*}\begin{aligned}u_0(t) & =\frac{1}{2}\int_{0}^{t}\([\til u(-s)\til v(t)+\til v(-s)\til u(t)]h(s)\right. \\ & \left.+[\til u(s)\til v(t)-\til v(s)\til u(t)]h(-s)\)\dif s,\end{aligned}\quad\sfa t\in\bR,\end{equation*}
 lo cual demuestra el resultado.
\end{proof}
\subt{Aplicaciones}\\
En esta sección utilizamos las expresiones anteriores para obtener la expresión explícita de la función de Green dependiendo de los valores de las constantes $a$ y $b$. Además, estudiamos el signo de la función y deducimos los principios de comparación correspondientes.

Separamos el estudio en tres casos, teniendo en consideración la expresión de la solución general de la ecuación \eqref{heabconst}.

\ssubt{El caso (C1)}\\
Usando la ecuación \eqref{eqsolgen}, obtenemos la siguiente expresión de $G$ para el caso (C1).

\begin{align*}G(t,s)= &\[\cos(\omega (s - t))+ \frac{b}{\omega}\sen(\omega(s-t))\]\chi_0^t(s) \\ & +\frac{a}{\omega}\sen(\omega(s+t))\chi_{-t}^0(s).\end{align*}
Lo cual podemos rescribir como
\begin{subequations}
\begin{empheq}[left={G(t,s)=\empheqlbrace}]{align}
   &  \cos\omega (s - t) + \frac{b}{\omega}\sen\omega(s-t), &  0\le s \le t, \label{G1} \\
     & \label{G2} -\cos\omega (s - t) - \frac{b}{\omega}\sen\omega(s-t), & t\le s \le 0, \\ 
     & \frac{a}{\omega}\sen\omega(s+t), &  -t\le s \le 0, \label{G3} \\
     & \label{G4} -\frac{a}{\omega}\sen\omega(s+t), &  0\le s \le -t, \\
     &      0 & \text{en otro caso.}
 \end{empheq}
\end{subequations}
Estudiando la expresión de $G$ podemos obtener principios del máximo y del antimáximo que nos darán información sobre el signo de la solución, como ilustra el siguiente resultado.

\begin{lem}\label{lempma1} Supongamos que $a^2>b^2$ y definamos
$$\eta(a,b):=
\left\{
\begin{array}{lll}
\frac{1}{\sqrt{a^2-b^2}}\arctan \frac{\sqrt{a^2-b^2}}{b}, & \mbox{si}  & b>0,\\
\frac{\pi}{2|a|}, & \mbox{si}  & b=0,\\
\frac{1}{\sqrt{a^2-b^2}}\left(\arctan \frac{\sqrt{a^2-b^2}}{b}+\pi\right), & \mbox{si} & b<0.
\end{array}
\right.$$

Entonces la función de Green del problema \eqref{gpabconst} es
 \begin{itemize}
 \item  positiva en $\{(t,s), \; 0<s<t\}$ si y solo si $t \in (0,\eta(a,b))$,
 \item  negativa en $\{(t,s), \; t<s<0\}$ si y solo si $t \in (-\eta(a,-b),0)$.
 \end{itemize}
 Si $a>0$, la función de Green del problema \eqref{gpabconst} es
\begin{itemize}
\item  positiva en $\{(t,s), \; -t<s<0\}$ si y solo si $t \in (0,\pi/ \sqrt{a^2-b^2})$,
\item  positiva en $\{(t,s), \; 0<s<-t\}$ si y solo si $t \in (-\pi/ \sqrt{a^2-b^2},0)$,
\end{itemize}
y, si $a<0$, la función de Green del problema \eqref{gpabconst} es
\begin{itemize}
\item  negativa en $\{(t,s), \; -t<s<0\}$ si y solo si $t \in (0,\pi/ \sqrt{a^2-b^2})$,
\item  negativa en $\{(t,s), \; 0<s<-t\}$ si y solo si $t \in (-\pi/ \sqrt{a^2-b^2},0)$.
\end{itemize}
\end{lem}
\begin{proof}Para $0<b<a$, el argumento del seno en \eqref{G3} es positivo, luego \eqref{G3} es positiva para $t<\pi/\omega$. Por otra banda, es fácil comprobar que \eqref{G1} es positiva siempre y cuando $t<\eta(a,b)$.\par
El resto de la demostración continúa de manera similar.
\end{proof}

Como corolario del resultado anterior tenemos el siguiente.

\begin{lem}\label{lempma1b} Supongamos que $a^2>b^2$. Entonces,
\begin{itemize}
\item si $a>0$, la función de Green del problema \eqref{gpabconst} es positiva en $[0,\eta(a,b)]\times\bR$,
\item si $a<0$, la función de Green del problema \eqref{gpabconst} es negativa en $[-\eta(a,-b),0]\times\bR$,
\item la función de Green del problema \eqref{gpabconst} cambia de signo en cualquier otra banda que no sea un subconjunto de las anteriores.
\end{itemize}
\end{lem}
\begin{proof} La demostración se deriva del resultado anterior junto con el hecho de que
$$\eta(a,b)\le\frac{\pi}{2\omega}<\frac{\pi}{\omega}.$$
\end{proof}

\ssubt{El caso (C2)}\\
Estudiamos aquí el caso (C2). Queda claro que
\begin{align*}G(t,s)  = & \[\cosh(\omega (s - t)) + \frac{b}{\omega}\senh(\omega(s-t))\]\chi_0^t(s) \\ & +\frac{a}{\omega}\senh(\omega(s+t))\chi_{-t}^0(s),\end{align*}
lo cual se puede rescribir como
\begin{subequations}
\begin{empheq}[left={G(t,s)=\empheqlbrace}]{align}
   &  \cosh\omega (s - t) + \frac{b}{\omega}\senh\omega(s-t), &  0\le s \le t, \label{G1b} \\
     & \label{G2b} -\cosh\omega (s - t) - \frac{b}{\omega}\senh\omega(s-t),\hspace*{-5mm} & t,\le s \le 0, \\ 
     & \frac{a}{\omega}\senh\omega(s+t), &  -t\le s \le 0, \label{G3b} \\
     & \label{G4b} -\frac{a}{\omega}\senh\omega(s+t), &  0\le s \le -t, \\
       &    0 & \text{en otro caso.}
 \end{empheq}
\end{subequations}
Estudiando la expresión de $G$ podemos obtener principios del máximo y del antimáximo como hicimos para el caso (C1).

\begin{lem}\label{lempma2} Supongamos que $a^2<b^2$ y definamos
$$\sigma(a,b):=\frac{1}{\sqrt{b^2-a^2}}\arctanh \frac{\sqrt{b^2-a^2}}{b}.$$
Entonces, 
\begin{itemize}
\item si $a>0$, la función de Green del problema \eqref{gpabconst} es positiva en $\{(t,s), \; -t<s<0\}$ y $\{(t,s), \; 0<s<-t\}$,
\item  si $a<0$, la función de Green del problema \eqref{gpabconst} es negativa en $\{(t,s), \; -t<s<0\}$ y $\{(t,s), \; 0<s<-t\}$,
\item si $b>0$, la función de Green del problema \eqref{gpabconst} es negativa en $\{(t,s), \; t<s<0\}$,
\item si $b>0$, la función de Green del problema \eqref{gpabconst} es positiva en $\{(t,s), \; 0<s<t\}$ si y solo si $t \in (0,\sigma(a,b))$,
\item  si $b<0$, la función de Green del problema \eqref{gpabconst} es positiva en $\{(t,s), \; 0<s<t\}$,
\item  si $b<0$, la función de Green del problema \eqref{gpabconst} es negativa en $\{(t,s), \; t<s<0\}$ si y solo si $t \in (\sigma(a,b),0)$.

\end{itemize}
\end{lem}
\begin{proof}Para $0<a<b$, el argumento del seno hiperbólico en \eqref{G4} es negativo, por lo tanto \eqref{G4b} es positiva. El argumento de $\sinh$ en \eqref{G3b} es positivo, luego \eqref{G3b} es positiva. Resulta fácil comprobar que \eqref{G1b} es positiva siempre y cuando $t<\sigma(a,b)$.\par
Por otra parte, \eqref{G2b} es siempre negativa.\par
El resto de la demostración continúa de manera similar.
\end{proof}

Como corolario del anterior resultado podemos probar el siguiente.

\begin{lem}\label{lempma} Supongamos que  $a^2<b^2$. Entonces,
\begin{itemize}
\item si $0<a<b$, la función de Green del problema \eqref{gpabconst} es positiva en $[0,\sigma(a,b)]\times\bR$,
\item si $b<-a<0$, la función de Green del problema \eqref{gpabconst} es positiva en $[0,+\infty)\times\bR$,
\item si $b<a<0$, la función de Green del problema \eqref{gpabconst} es negativa en $[\sigma(a,b),0]\times\bR$,
\item si $b>-a>0$, la función de Green del problema \eqref{gpabconst} es negativa en $(-\infty,0]\times\bR$,
\item la función de Green del problema \eqref{gpabconst} cambia de signo en cualquier banda que no sea un subconjunto de las anteriores.
\end{itemize}
\end{lem}
\ssubt{El caso (C3)}\\
Ahora consideramos el caso (C3), primero para $a=b$. En este caso,
$$G(t,s)=[1+a(s-t)]\chi_0^t(s)+a(s+t)\chi_{-t}^0(s),$$
lo cual podemos rescribir como
$$G(t,s)=\begin{cases}
  1+a(s-t), &  0\le s \le t, \\
     -1-a(s-t), & t\le s \le 0, \\ 
   a(s+t), &  -t\le s \le 0, \\
      -a(s+t), &  0\le s \le -t, \\
      0 & \text{en otro caso.}
\end{cases}$$
Estudiando la expresión de $G$ podemos obtener principios del máximo y del antimáximo.
Con esta información, probaremos el siguiente resultado análogo a los otros previamente demostrados para los casos (C1) y (C2).

\begin{lem}\label{lempma3} Supongamos que $a=b$. Entonces, si $a>0$, la función de Green del problema \eqref{gpabconst} es
\begin{itemize}
\item positiva en $\{(t,s), \; -t<s<0\}$ y $\{(t,s), \; 0<s<-t\}$,
\item negativa en $\{(t,s), \; t<s<0\}$,
\item positiva en $\{(t,s), \; 0<s<t\}$ si y solo si $t \in (0,1/a)$,
\end{itemize}
y, si $a<0$, la función de Green del problema \eqref{gpabconst} es
\begin{itemize}
\item negativa en $\{(t,s), \; -t<s<0\}$ y $\{(t,s), \; 0<s<-t\}$,
\item positiva en $\{(t,s), \; 0<s<t\}$.
\item negativa en $\{(t,s), \; t<s<0\}$ si y solo si  $t \in (1/a,0)$.
\end{itemize}
\end{lem}

Como corolario del resultado anterior tenemos el siguiente.

\begin{lem}\label{lempma3b} Supongamos que $a=b$. Entonces,
\begin{itemize}
\item si $a>0$, la función de Green del problema \eqref{gpabconst} es positiva en $[0,1/a]\times\bR$,
\item si $a<0$, la función de Green del problema \eqref{gpabconst} es negativa en $[1/a,0]\times\bR$,
\item la función de Green del problema \eqref{gpabconst} cambia de signo en cualquier banda que no sea un subconjunto de las anteriores.
\end{itemize}

\end{lem}

Para este caso particular tenemos además otra forma de calcular la solución del problema que, como ya sabemos, es única.

\begin{pro} Sea $a=b$  y asumamos que $2at_0\ne1$. Sea $H(t):=\int_{t_0}^th(s)\dif s$ y $\cH(t):=\int_{t_0}^t H(s)\dif s$. Entonces el problema \eqref{gpabconst} tiene una solución única dada por
$$u(t)=H(t)-2a\cH_o(t)+\frac{2a\,t-1}{2a\,t_0-1}c.$$
\end{pro}
\begin{proof} La ecuación se satisface, dado que
\begin{align*} & u'(t)+a(u(t)+u(-t))= u'(t)+2a u_e(t) \\= & h(t)-2\,a H_e(t) +\frac{2a\,c}{2a\,t_0-1}+2\, a H_e(t)-\frac{2a\,c}{2a\,t_0-1}=h(t).\end{align*}
La condición inicial también se satisface dado que $u(t_0)=c$.
\end{proof}
El caso (C3.2) es muy similar, con
$$G(t,s)=\begin{cases}
  1+a(t-s), &  0\le s \le t, \\
     -1-a(t-s), & t\le s \le 0, \\ 
   a(s+t), &  -t\le s \le 0, \\
      -a(s+t), &  0\le s \le -t, \\
      0 & \text{en otro caso.}
\end{cases}$$
\begin{lem}\label{lempma5} Supongamos que $a=-b$, $u_0(t_0)=c$, $\til u(t_0)\ne0$ y $h\ge0$. Entonces,
\begin{itemize}
\item si $a>0$, la solución del problema \eqref{gpabconst} es positiva en $[0,+\infty)$,
\item si $a<0$, la solución del problema \eqref{gpabconst} es negativa en $(-\infty,0]$,
\end{itemize}
\end{lem}
De nuevo, para este caso particular, tenemos otra forma de calcular la solución del problema.
\begin{pro} Sean $a=b$, $H(t):=\int_{0}^th(s)\dif s$, $\cH(t):=\int_{0}^t H(s)\dif s$. Entonces el problema \eqref{gpabconst} tiene una única solución dada por
$$u(t)=H(t)-H(t_0)-2a(\cH_e(t)-\cH_e(t_0))+c.$$
\end{pro}
\begin{proof}La ecuación se satisface, ya que
\begin{align*} & u'(t)+\l(u(t)-u(-t))=u'(t)+2\l u_o(t)\\ = & h(t)-2\l H_o(t)+2\l H_o(t)=h(t).\end{align*}
También se cumple la condición inicial, pues resulta claro que $u(t_0)=c$.
\end{proof}

\subt{Problemas con condiciones de contorno}\\
Consideraremos ahora problemas con reflexión y condiciones de contorno periódicas. En particular el siguiente problema (ver \cite{Cab4}).
\begin{equation}\label{eqoriginal}x'(t)+\omega x(-t)=h(t),\ t\in I,\quad x(-T)=x(T),\end{equation}
donde $T\in\bR^+$, $\omega\in\bR\backslash\{0\}$ y $h\in L^1(I)$, con $I:=[-T,T]$. La forma exacta de la función de Green viene dada por el siguiente Teorema.
\begin{thm}[\cite{Cab4}, Proposición 3.2]\label{Greenf} Supongamos que $\omega \neq k \, \pi/T$, $k \in \bZ$. Entonces el problema (\ref{eqoriginal}) tiene una única solucion dada por la expresión
\begin{equation*}
\label{e-u}
u(t):=\int_{-T}^T\ol G(t,s)h(s)\dif s,
\end{equation*}
donde $$\ol{G}(t,s):=\omega\,G(t,-s)-\frac{\partial G}{\partial s}(t,s)$$
y $G$ es la función de Green para la ecuación del oscilador armónico con condiciones de contorno homogéneas,
$$x''(t)+\omega^2x(t)=0;\ x(T)=x(-T),\ x'(T)=x'(-T).$$
\end{thm}
Las propiedades del signo de esta función son estudiadas en más detalle en \cite{Cab5}, donde se usan métodos similares a los de \cite{gijwjiea, gijwjmaa, gijwems} para obtener resultados de existencia y multiplicidad.\par
En lo que queda de trabajo exploraremos este caso y obtendremos funciones de Green para problemas con coeficientes no constantes. Además demostraremos el Teorema de Correspondencia de Involuciones, con el que seremos capaces de aplicar estos resultados a problemas con otras involuciones. Por último, obtendremos también nuevos principios del máximo y del antimáximo y resultados de existencia y unicidad de soluciones.

\subt{Equivalencia de involuciones}\\
Siguiendo \cite{Cab7}, supongamos ahora que $\phi$ es una involución diferenciable en $[\tau_1,\tau_2]$. Sea $a,b,c,d\in L^1([\tau_1,\tau_2])$ y consideremos el siguiente problema
\begin{equation}\begin{aligned}\label{proinv1} d(t)x'(t)+c(t)x'(\phi(t))+b(t)x(t)+a(t)x(\phi(t)) & =h(t), \\ x(\tau_1) & =x(\tau_2).
\end{aligned}\end{equation}
Sería interesante conocer bajo qué circunstancias el problema \eqref{proinv1} es equivalente a otro problema del mismo tipo pero con una involución diferente, en particular una reflexión. Los siguientes resultados aclararán esta situación.
\begin{lem}[\textsc{Correspondencia de involuciones}]\label{corofinv} Sea $\phi$ y $\psi$ dos involuciones diferenciables\footnote{Toda involución diferenciable es un difeomorfismo.} definidas en sendos intervalos $[\tau_1,\tau_2]$ y $[\sigma_1,\sigma_2]$ respectivamente. Sean $t_0$ y $s_0$ los únicos puntos fijos de $\phi$ y $\psi$ respectivamente. Entonces, existe un difeomorfismo $f:[\sigma_1,\sigma_2]\to[\tau_1,\tau_2]$ que preserva la orientación y tal que $f(\psi(s))=\phi(f(s))\sfa s\in[\sigma_1,\sigma_2]$.
\end{lem}

\begin{proof}
Sea $g:[\sigma_1,s_0]\to[\tau_1,t_0]$ un difeomorfismo que preserva la orientación, esto es, $g(s_0)=t_0$. Definamos
$$f(s):=\begin{cases} g(s) & \text{ si } s\in[\sigma_1,s_0], \\ (\phi\circ g\circ\psi)(s) & \text{ si } s\in(s_0,\sigma_2].\end{cases}$$
\par
Claramente,  $f(\psi(s))=\phi(f(s))\sfa s\in[\sigma_1,\sigma_2]$. Puesto que $s_0$ es un punto fijo de $\psi$, $f$ es continua. Es más, ya que $\phi$ y $\psi$ son involuciones, $\phi'(t_0)=\psi'(s_0)=-1$, luego $f$ es diferenciable. $f$ es invertible de inversa
$$f^{-1}(t):=\begin{cases} g^{-1}(t) & \text{ si } t\in[\tau_1,t_0], \\ (\psi\circ g^{-1}\circ\phi)(t) & \text{ si } t\in(t_0,\tau_2].\end{cases}$$
$f^{-1}$ es también diferenciable por las mismas razones, con lo que queda demostrado el Lema.
\end{proof}
\begin{rem}Un argumento similar puede realizarse en el caso de involuciones definidas en intervalos abiertos y posiblemente no acotados.
\end{rem}
\begin{rem} La expresión obtenida para $f$ nos recuerda a la caracterización de las involuciones dada en \cite[Propiedad 6]{Wie}.
\end{rem}
\begin{rem} Resulta fácil comprobar que si $\phi$ es una involución definida en $\bR$ con punto fijo $t_0$, entonces $\psi(t):=\phi(t+t_0-s_0)-t_0+s_0$ es una involución con punto fijo $s_0$ (cf. \cite[Propiedad 2]{Wie}). Para estas elecciones particulares de $\phi$ y $\psi$, podemos tomar $g(s)=s-s_0+t_0$ en el Lema \ref{corofinv} y, en tal caso, $f(s)=s-s_0+t_0$ para todo $s\in\bR$.
\end{rem}

\begin{cor}[\textsc{Cambio de involución}] Bajo las hipótesis del Lema \ref{corofinv},
el problema \eqref{proinv1} es equivalente a
\begin{equation}\label{proinv2}\begin{aligned} \frac{d(f(s))}{f'(s)}y'(s)+\frac{c(f(s))}{f'(\psi(s))}y'(\psi(s)) \\  +b(f(s))y(s)+a(f(s))y(\psi(s)) & =h(f(s)), \\ y(\sigma_1) & =y(\sigma_2).
\end{aligned}\end{equation}
\end{cor}
\begin{proof}Consideremos el cambio de variable $t=f(s)$ e $y(s):=x(t)=x(f(s))$. Entonces, usando el Lema \ref{corofinv}, resulta claro que
$$\frac{\dif y}{\dif s}(s)=\frac{\dif x}{\dif t}(f(s))\frac{\dif f}{\dif s}(s)\quad\text{y}\quad \frac{\dif y}{\dif s}(\psi(s))=\frac{\dif x}{\dif t}(\phi(f(s)))\frac{\dif f}{\dif s}(\psi(s)).$$
Haciendo las sustituciones apropiadas en el problema \eqref{proinv1} obtenemos el problema \eqref{proinv2} y vice-versa.
\end{proof}
\begin{rem} El Corolario anterior del Teorema del Cambio de Involución se puede rescribir de manera inmediata para el caso del problema de valor inicial. La razón por la que no fue enunciado en su momento es que, un problema con coeficientes constantes, tras el cambio de involución, puede pasar a un problema con coeficientes variables.
\end{rem}
Este último resultado nos permite restringir el estudio del problema \eqref{proinv1} al caso en que $\phi$ es una reflexión $\phi(t)=-t$. En la próxima sección limitaremos algo más las hipótesis al caso en que $c\equiv0$ en el problema \eqref{proinv1}, lo cual no supone una pérdida de generalidad, ya que mediante un cambio de variable se puede obtener esta condición (ver Apéndice en \cite{Cab7}).

\subt{Estudio de la ecuación homogénea}\\
En esta sección estudiaremos los distintos casos de la ecuación homogénea
\begin{equation}\label{gen-eq} x'(t)+a(t)x(-t)+b(t)x(t)=0,\end{equation}
donde $a,b\in L^1(I)$. Estos casos guardarán, como veremos, similitud con los (C1)-(C3) estudiados para el problema con condición inicial, sin embargo hay una diferencia fundamental en el hecho de que los coeficientes no sean constantes. Para resolver dicha ecuación, podemos considerar la descomposición de la ecuación \eqref{gen-eq} usada en \cite{Cab4}.
\begin{align}\label{eq3.1}\begin{pmatrix}x_o' \\ x_e'\end{pmatrix} & =\begin{pmatrix}a_o-b_o & -a_e-b_e \\ a_e-b_e & -a_o-b_o\end{pmatrix}\begin{pmatrix}x_o \\ x_e\end{pmatrix}.
\end{align}
Obsérvese que, a priori, las soluciones del sistema \eqref{eq3.1} no tienen por qué ser pares de funciones pares e impares, ni tienen por qué proporcionar soluciones de la ecuación \eqref{gen-eq}.\par
Para resolver el sistema, restringiremos el problema a aquellos casos en los que la matriz
$$M(t)=\begin{pmatrix}a_o-b_o & -a_e-b_e \\ a_e-b_e & -a_o-b_o\end{pmatrix}(t)$$
satisfaga que $[M(t),M(s)]:=M(t)M(s)-M(s)M(t)=0\sfa t,s\in I$, dado que, solamente en este caso, la solución del sistema  \eqref{eq3.1} viene dada por la exponencial de la integral de $M$ \cite{Kot}. Claramente,
\begin{align*} & [M(t),M(s)]=\\ & 2\(\begin{smallmatrix} a_e(t)b_e(s)-a_e(s)b_e(t) & a_o(s)[a_e(t)+b_e(t)]-a_o(t)[a_e(s)+b_e(s)]\\a_o(t)[a_e(s)+b_e(s)]-a_o(s)[a_e(t)+b_e(t)] & a_e(s)b_e(t)-a_e(t)b_e(s)\end{smallmatrix}\).\end{align*}
Sean $A(t):=\int_0^t a(s)\dif s$, $B(t):=\int_0^t b(s)\dif s$ y $\ol M$ una primitiva (excepto tal vez una matriz constante) de $M$. Estudiamos ahora los distintos casos en los que $[M(t),M(s)]=0\sfa t,s\in I$. Asumiremos siempre $a\not\equiv0$, puesto que el caso $a\equiv0$ es el de una EDO bien conocida.\par
\textbf{(C1'). $b_e=k\,a,\ k\in\bR,\ |k|<1$.}
\par
En este caso, $a_o=0$ y $\ol M$ tiene la forma
$$\ol M=\begin{pmatrix}B_e &  -(1+k)A_o\\ (1-k)A_o & -B_e\end{pmatrix}.$$
Si calculamos la exponencial obtenemos
\begin{align*} & e^{\ol M(t)}= \\ & e^{-B_e(t)}\begin{pmatrix}\cos\(\sqrt{1-k^2}A(t)\) & -\frac{1+k}{\sqrt{1-k^2}}\sen\(\sqrt{1-k^2}A(t)\)\\ \frac{\sqrt{1-k^2}}{1+k}\sen\(\sqrt{1-k^2}A(t)\) & \cos\(\sqrt{1-k^2}A(t)\)\end{pmatrix}.\end{align*}
Por lo tanto, si una solución de la ecuación \eqref{gen-eq} existe, tiene que ser de la forma
\begin{align*}u(t)= & \a e^{-B_e(t)}\cos\(\sqrt{1-k^2}A(t)\)\\ & +\b e^{-B_e(t)}\frac{1+k}{\sqrt{1-k^2}}\sen\(\sqrt{1-k^2}A(t)\).\end{align*}
con $\a$, $\b\in\bR$. Resulta sencillo comprobar que todas las soluciones de la ecuación \eqref{gen-eq} son de esta forma con $\b=-\a$.
\par
\textbf{(C2'). $b_e=k\,a,\ k\in\bR,\ |k|>1$.}
Este caso es muy parecido al (C1') y proporciona soluciones del sistema \eqref{eq3.1} de la forma
\begin{align*}u(t)= & \a e^{-B_e(t)}\cosh\(\sqrt{k^2-1}A(t)\)\\ & +\b e^{-B_e(t)}\frac{1+k}{\sqrt{k^2-1}}\senh\(\sqrt{k^2-1}A(t)\),\end{align*}
que son soluciones de la ecuación \eqref{gen-eq} cuando $\b=-\a$.\par
\textbf{(C3'). $b_e=a$.}
En este caso las soluciones del sistema \eqref{eq3.1} son de la forma
\begin{equation}\label{eqc3}u(t)=\a e^{-B_e(t)}+2\b e^{-B_e(t)}A(t)\end{equation}
las cuales son soluciones de la ecuación \eqref{gen-eq} cuando $\b=-\a$.\par
\textbf{(C4'). $b_e=-a$.}
En este caso las soluciones del sistema \eqref{eq3.1} son las mismas que en el (C3'), pero esta vez son soluciones de \eqref{gen-eq} cuando $\b=0$.\par
\textbf{(C5'). $b_e=a_e=0$.}
En esta última instancia, las soluciones de \eqref{eq3.1} son de la forma
$$u(t)=\a e^{A(t)-B(t)}+\b e^{-A(t)-B(t)},$$
verificando la ecuación \eqref{gen-eq} cuando $\a=0$.\par

\subt{Los casos (C1')--(C3') para el problema completo}\\
Sea $T\in\bR^+$, en la situación más general del siguiente problema
\begin{equation}\label{eq2cp} x'(t)+a(t)\,x(-t)+b(t)\,x(t)=h(t),\nkp t\in I,\quad x(-T)=x(T),
\end{equation}
los casos (C1')--(C3') pueden tratarse con relativa facilidad. De hecho, incluso podemos obtener la función de Green para el operador en cuestión.\par
Para empezar, probaremos la siguiente generalización del Teorema \ref{Greenf}.\par
Consideremos el problema \eqref{eq2cp} con $a$ y $b$ constantes.
\begin{equation}\label{eq2cp2} x'(t)+a\,x(-t)+b\,x(t)=h(t),\nkp t\in I,\quad x(-T)=x(T).
\end{equation}
Si consideramos ahora el caso homogéneo ($h\equiv0$), derivando y haciendo las sustituciones adecuadas, llegamos al problema
\begin{equation}\label{eqhog} \begin{aligned} & x''(t)+(a^2-b^2)x(t)=0,\nkp t\in I,\\ & x(-T)=x(T),\quad x'(-T)=x'(T).\end{aligned}
\end{equation}
el cual, para $b^2<a^2$, es el problema del oscilador armónico con condiciones de contorno homogéneas.

De manera análoga a como se hizo en \cite{Cab4} para el caso $b=0$, podemos demostrar el siguiente teorema.
\begin{thm}\label{Greenf2} Supongamos que $a^2-b^2 \neq n^2 \, (\pi/T)^2$, $n=0,1,\dots$ Entonces el problema \eqref{eq2cp2} tiene una única solución dada por la expresión
\begin{equation}
\label{e-u2}
u(t):=\int_{-T}^T\ol G(t,s)h(s)\dif s,
\end{equation}
donde $$\ol{G}(t,s):=a\,G(-t,s)-b\,G(t,s)+\frac{\partial G}{\partial t}(t,s)$$
y $G$ es la función de Green asociada al problema \eqref{eqhog}.
\end{thm}
\begin{proof}
Puesto que el problema \eqref{eq2cp2} puede ser reducido a uno con la ecuación de \eqref{eqhog}, la teoría clásica de EDO nos permite afirmar que \eqref{eq2cp2} tiene a lo sumo una solución para todo $a^2-b^2 \neq n^2 \, (\pi/T)^2$, $n=0,1,\dots$ Veamos ahora que la función $u$ definida en (\ref{e-u2}) satisface la ecuación (\ref{eq2cp2}).

\begin{align*}
 & u'(t)+a\, u(-t)+b\, u(t)=\frac{\dif}{\dif t}\int_{-T}^{-t}\ol G(t,s)h(s)\dif s+\frac{\dif}{\dif t}\int_{-t}^t\ol G(t,s)h(s)\dif s\\
 &+\frac{\dif}{\dif t}\int_{t}^T\ol G(t,s)h(s)\dif s+a\int_{-T}^T\ol G(-t,s)h(s)\dif s+b\int_{-T}^T\ol G(t,s)h(s)\dif s\\
=&(\ol G(t,t^-)-\ol G(t,t^+))h(t)\\ &+\int_{-T}^T\left[-a\frac{\partial G}{\partial t}(-t,s)-b\frac{\partial G}{\partial t}(t,s)+\frac{\partial^2 G}{\partial t^2}(t,s)\right]h(s)\dif s\\
& +a\int_{-T}^T\left[a\,G(t,s)-b\,G(-t,s)+\frac{\partial G}{\partial t}(-t,s)\]h(s)\dif s
\\&+b\int_{-T}^T\left[a\,G(-t,s)-b\,G(t,s) +\frac{\partial G}{\partial t}(t,s)\]h(s)\dif s\\
= & \[\frac{\partial G}{\partial t}(t,t^-)-\frac{\partial G}{\partial t}(t,t^+)\]h(s)+\int_{-T}^T\[\frac{\partial^2 G}{\partial t^2}(t,s)+(a^2-b^2)G(t,s)\]\dif s\\ = & h(s),
\end{align*}
teniéndose esta última igualdad por la propiedades de $G$ como función de Green del problema \eqref{eqhog}.

Podemos verificar además, usando las propiedades de $G$, las condiciones de contorno.
\begin{align*}& u(T)-u(-T)\\= & \int_{-T}^T\left[(a+b)\[G(-T,s)-\,G(T,s)\]+\frac{\partial G}{\partial t}(T,s)-\frac{\partial G}{\partial t}(-T,s)\right]h(s)\dif s\\  = & 0.\end{align*}
\end{proof}
\begin{rem}
Obsérvese que la propiedad $a^2-b^2 \neq n^2 \, (\pi/T)^2$, $n=0,1,\dots$ en el Teorema anterior implica, en particular, que $a\ne b$, lo cual descarta el caso (C3') para el problema \eqref{eq2cp2}.
\end{rem}
Este último Teorema nos lleva a la pregunta "¿Cuál es la función de Green para el problema \eqref{eq2cp2} en el caso (C3') con $a$ y $b$ constantes?". El siguiente Lema responde a esta cuestión.
\begin{lem}\label{lemc3}Sea $a\ne 0$ una constante y sea $G_{C3}$ Una función real definida como
$$G_{C3}(t,s):=\frac{t-s}{2}-a\,s\,t+\begin{cases} -\frac{1}{2}+a\,s & \text{ si } |s|<t, \\ \frac{1}{2}-a\,s & \text{ si } |s|<-t, \\ \frac{1}{2}+a\,t & \text{ si } |t|<s, \\ -\frac{1}{2}-a\,t & \text{ si } |t|<-s.\end{cases}$$
Entonces las siguientes propiedades se satisfacen.
\begin{itemize}
\item $\frac{\partial G_{C3}}{\partial t}(t,s)+a(G_{C3}(t,s)+G_{C3}(-t,s))=0$ para casi todo $t,s\in (-1,1)$.
\item $\frac{\partial G_{C3}}{\partial t}(t,t^+)-\frac{\partial G_{C3}}{\partial t}(t,t^-)=1\sfa t\in(-1,1)$.
\item $G_{C3}(-1,s)=G_{C3}(1,s)\sfa s\in(-1,1)$.
\end{itemize}
\end{lem}
Las propiedades anteriores se pueden comprobar de manera directa e implican que $G_{C3}$ es la función de Green para el problema \eqref{eq2cp2} en el caso en que $T=1$. Para obtener la función de Green en el caso general basta un sencillo cambio de variable.
\begin{rem} La función $G_{C3}$ puede ser obtenida a partir de la función de Green para el caso $(C1')$ tomando el límite $k\to 1^-$ para $T=1$.
\end{rem}
Consideremos ahora las siguientes condiciones.\par
$\mathbf{(C1^*)}$. (C1') se satisface, $(1-k^2)A(T)^2\neq (n \, \pi)^2$ para todo $n=0,1,\dots$ y $\cos\(\sqrt{1-k^2}A(T)\)\ne0$.\par
$\mathbf{(C2^*)}$. (C2') se satisface y $(1-k^2)A(T)^2\neq (n \, \pi)^2$  para todo $n=0,1,\dots$\par
$\mathbf{(C3^*)}$. (C3') se satisface y $A(T)\ne0$.\par
Supongamos que se satisface una de las condiciones $(C1^*)$--$(C3^*)$. En este caso, por el Teorema \ref{Greenf2}, podemos garantizar la unicidad de solución del siguiente problema \cite{Cab4}.
\begin{equation}\begin{aligned}\label{eq2} x'(t)+x(-t)+k\,x(t) & =h(t),\nkp t\in [-|A(T)|,|A(T)|],\\ x(A(T)) & =x(-A(T)).\end{aligned}
\end{equation}

La función de Green $G_2$ para el problema \eqref{eq2} no es más que un caso particular de $\ol G$ y se puede expresar como
$$\ol{G_2}(t,s):=\begin{cases} k_1(t,s),& t>|s|,\\
k_2(t,s),   & s>|t|,\\
k_3(t,s), & -t>|s|,\\
k_4(t,s),  & -s>|t|. \end{cases}$$
Definamos ahora
\begin{equation}\begin{aligned}\label{Ggenral} & G_1(t,s):=e^{B_e(s)-B_e(t)}H(t,s)\\ = & e^{B_e(s)-B_e(t)}\begin{cases} k_1(A(t),A(s)),& t>|s|,\\
k_2(A(t),A(s)),   & s>|t|,\\
k_3(A(t),A(s)), & -t>|s|,\\
k_4(A(t),A(s)),  & -s>|t|. \end{cases}\end{aligned}\end{equation}
Definida de esta manera, $G_1$ es continua salvo en la diagonal, donde $G_1(t,t^-)- G_1(t,t^+)=1$. Ahora podemos establecer el siguiente teorema.
\begin{thm}\label{thmcases123}
Asumamos que se cumple una de las condiciones $(C1^*)$--$(C3^*)$. Sea $G_1$ definida como en \eqref{Ggenral}. Supongamos que $G_1(t,\cdot)h(\cdot)\in L^1(I)$ para cada $t\in I$. Entonces el problema \eqref{eq2cp} tiene solución única dada por
$$u(t)=\int_{-T}^TG_1(t,s)h(s)\dif s.$$
\end{thm}
\begin{proof}
Obsérvese que, dado que $a$ es par, $A$ es impar, por lo tanto $A(-t)=-A(t)$. Es importante observar que si $a$ no tiene signo constante en $I$, entonces $A$ puede no ser inyectiva en $I$.

De las propiedades de $\bar G_2$ como función de Green, resulta claro que
$$\frac{\partial \bar G_2}{\partial t}(t,s)+\bar G_2(-t,s)+k\,\bar G_2(t,s)=0\quad\text{para casi todo }t,s\in A(I),$$
y por lo tanto,
$$\frac{\partial H}{\partial t}(t,s)+a(t)H(-t,s)+ka(t)\,H(t,s)=0\quad\text{para casi todo }t,s\in I,$$
En consecuencia

\begin{align*}
& u'  (t)+a(t)\, u(-t)+(b_o(t)+k\,a(t))\, u(t) \\= & \frac{\dif}{\dif t}\int_{-T}^TG_1(t,s)h(s)\dif s+a(t)\int_{-T}^TG_1(-t,s)h(s)\dif s\\
&\quad+(b_o(t)+k\,a(t))\int_{-T}^TG_1(t,s)h(s)\dif s\\
=\ &\frac{\dif}{\dif t}\int_{-T}^{t}  e^{B_e(s)-B_e(t)}H(t,s)h(s)\dif s+\frac{\dif}{\dif t}\int_{t}^T e^{B_e(s)-B_e(t)} H(t,s)h(s)\dif s\\&\quad+ a(t)\int_{-T}^T e^{B_e(s)-B_e(t)}H(-t,s)h(s)\dif s\\ & \quad+(b_o(t)+k\,a(t))\int_{-T}^T e^{B_e(s)-B_e(t)}H(t,s)h(s)\dif s\\
=\ &[H(t,t^-)-H(t,t^+)]h(t)+a(t)\, e^{-B_e(t)}\int_{-T}^Te^{B_e(s)}\frac{\partial H}{\partial t}(t,s)h(s)\dif s\\&\quad- b_o(t)e^{-B_e(t)}\int_{-T}^Te^{B_e(s)}H(t,s)h(s)\dif s\\ & \quad+ a(t)e^{-B_e(t)}\int_{-T}^Te^{B_e(s)} H(-t,s)h(s)\dif s\\ &\quad+ (b_o(t)+k\,a(t))e^{-B_e(t)}\int_{-T}^Te^{B_e(s)}H(t,s)h(s)\dif s\\ =\ & h(t)+\, a(t)e^{-B_e(t)}\int_{-T}^Te^{B_e(s)}\bigg[\frac{\partial H}{\partial t}(t,s)+a(t)H(-t,s)\\ &\quad+ka(t)\,H(t,s)\bigg]h(s)\dif s=  h(t).
\end{align*}

Las condiciones de contorno también se satisfacen.
$$u(T)-u(-T)= e^{-B_e(T)}\int_{-T}^Te^{B_e(s)} [H(T,s)-H(-T,s)]h(s)\dif s=0.$$
Para demostrar la unicidad de solución, sean $u$ y $v$ soluciones del problema \eqref{eq2}. Entonces $u-v$ satisface la ecuación \eqref{gen-eq} y es por tanto de la forma considerada en los casos $(C1')$--$(C3')$. Por otra parte, $(u-v)(T)-(u-v)(-T)=2(u-v)_o(T)=0$, pero esto solo es posible, dadas las condiciones extra impuestas por las hipótesis $(C1^*)$--$(C3^*)$, si $u-v\equiv0$, demostrándose así la unicidad de solución.
\end{proof}
Una de las importantes consecuencias directas del Teorema \ref{thmcases123} es la deducción inmediata de principios del máximo y del antimáximo en el caso $b\equiv0$\footnote{Obsérvese que esto descarta el caso (C3'), para el cual $b\equiv0$ implica $a\equiv 0$, ya que estamos asumiendo $a\not\equiv0$.}. Para mostra esto y qué ocurre en el caso $b$ constante, $b\ne0$, recordamos aquí algunos resultados de \cite{Cab4}.
\begin{thm}[{\cite[Teorema 4.3]{Cab4}}]\label{alphasign}\ \par
\begin{itemize}
\item Si $\a\in(0,\frac{\pi}{4})$ entonces $\ol G$ es estrictamente positiva $I^2$.
\item Si $\a\in(-\frac{\pi}{4},0)$ entonces $\ol G$ es estrictamente negativa $I^2$.
\item Si $\a=\frac{\pi}{4}$ entonces $\ol G$ se anula en $$P:=\{(-T,-T),(0,0),(T,T),(T,-T)\}$$ y es estrictamente positiva en $(I^2)\backslash P$.
\item Si $\a=-\frac{\pi}{4}$ entonces $\ol G$ se anula en $P$ y es estrictamente negativa en $(I^2)\backslash P$.
\item Si $\a\in\bR\backslash[-\frac{\pi}{4},\frac{\pi}{4}]$ entonces $\ol G$ no es positiva ni negativa en $I^2$.
\end{itemize}
\end{thm}
\begin{dfn} Sea $\cF_\l(I)$ un conjunto de de funciones reales diferenciables $f$ definidas en $I$ tales que $f(-T)-f(T)=\l$. Un operador lineal $R:\cF_\l(I)\to L^1(I)$ se dice
\begin{itemize}
\item \textbf{fuertemente inverso positivo} en $\cF_\l(I)$ si $Rx(t)>0$ para casi todo $t\in I$ implica $x>0$ para todo $x\in\cF_\l(I)$.
\item \textbf{fuertemente inverso negativo} en $\cF_\l(I)$ si $Rx(t)>0$ para casi todo $t\in I$ implica $x<0$ para todo $x\in\cF_\l(I)$.
\end{itemize}
\end{dfn}
\begin{cor}[{\cite[Corolario 4.4]{Cab4}}]\label{coralphasign} Sea $\cF_\l(I)$ el conjunto de funciones diferenciables $f$ definidas en $I$ tales que $f(-T)-f(T)=\l$. El operador $R_m:\cF_\l(I)\to L^1(I)$ definido como $R_m(x(t))=x'(t)+m\, x(-t)$, con $m\in\bR\backslash\{0\}$, satisface
\begin{itemize}
\item $R_m$ es fuertemente inverso positivo si y solo si $m\in(0,\frac{\pi}{4T}]$ y $\l\ge0$,
\item $R_m$ es fuertemente inverso negativo si y solo si $m\in[-\frac{\pi}{4T},0)$ y $\l\ge0$.

\end{itemize}
\end{cor}
Con estos resultados obtenemos el siguiente Corolario al Teorema \ref{thmcases123}.
\begin{cor}\label{corsigng}Bajo las condiciones del Teorema \ref{thmcases123}, en el caso $b=0$,\par
\begin{itemize}
\item Si $A(T)\in(0,\frac{\pi}{4})$ entonces $G_1$ es estrictamente positiva en $I^2$.
\item Si $A(T)\in(-\frac{\pi}{4},0)$ entonces $G_1$ es estrictamente negativa en $I^2$.
\item Si $A(T)=\frac{\pi}{4}$ entonces $G_1$ se anula en $$P':=\{(-A(T),-A(T)),(0,0),(A(T),A(T)),( A(T),- A(T))\}$$ y es estrictamente positiva en $(I^2)\backslash P$.
\item Si $A(T)=-\frac{\pi}{4}$ entonces $G_1$ se anula en $P$ y es estrictamente negativa en $(I^2)\backslash P$.
\item Si $A(T)\in\bR\backslash[-\frac{\pi}{4},\frac{\pi}{4}]$ entonces $G_1$ no es ni positiva ni negativa en $I^2$.
\end{itemize}
Además, el operador $R_a:\cF_\l(I)\to L^1(I)$ definido como $R_a(x(t))=x'(t)+a(t)\, x(-t)$ satisface
\begin{itemize}
\item $R_a$ es fuertemente inverso positivo si y solo si $A(T)\in(0,\frac{\pi}{4T}]$ y $\l\ge0$,
\item $R_a$ es fuertemente inverso negativo si y solo si $A(T)\in[-\frac{\pi}{4T},0)$ y $\l\ge0$.

\end{itemize}
\end{cor}
La segunda parte de este último Corolario, obtenida a partir de la positividad (o negatividad) de la función de Green podría deducirse, como probamos más abajo, sin tener tanto conocimiento acerca de la función de Green. Para ver esto, considérese la siguiente proposición en la línea del trabajo de Torres \cite{Tor}.\par
\begin{pro}\label{proredpro}
Considérese el problema de valor inicial homogéneo
\begin{equation}\label{eqhomivp} x'(t)+a(t)\,x(-t)+b(t)\,x(t)=0,\ t\in I;\ x(t_0)=0.\end{equation}
Si el problema \eqref{eqhomivp} tiene solución única ($x\equiv 0$) en $I$ para todo $t_0\in I$ entonces, si la función de Green para \eqref{eq2cp} existe, tiene signo constante.\par
Es más, si además asumimos que $a+b$ tiene signo constante, la función de Green tiene el mismo signo que  $a+b$.
\end{pro}
\begin{proof}
Sin pérdida de generalidad, consideremos que $a$ es una función de $L^1(I)$ $2T$-periódica  definida en $\bR$ (la solución de \eqref{eq2cp} será considerada en $I$). Sea $G_1$ la función de Green del problema \eqref{eq2cp}. Dado que $G_1(T,s)=G_1(-T,s)$ para todo $s\in I$, y $G_1$ es continua salvo en la diagonal, es suficiente con demostrar que $G_1(t,s)\ne0\sfa t,s\in I$.\par
Supongamos, por el contrario, que existen $t_1,s_1\in I$ tales que $G_1(t_1,s_1)=0$. Sea $g$ la extensión $2T$-periódica de $G_1(\cdot,s_1)$. Supongamos que $t_1>s_1$ (el otro caso sería análogo). Sea $f$ la restricción de $g$ a $(s_1,s_1+2T)$. $f$ es absolutamente continua y satisface \eqref{eqhomivp} en c.t.p. para $t_0=t_1$, por tanto, $f\equiv0$. Esto contradice el hecho de que $G_1$ sea la función de Green y, como consecuencia, $G_1$ tiene que tener signo constante.\par
Observemos ahora que $x\equiv1$ satisface
$$x'(t)+a(t)x(-t)+b(t)x(t)=a(t)+b(t),\ x(-T)=x(T).$$
Por tanto, $\int_{-T}^TG_1(t,s)(a(t)+b(t))\dif s=1$ para todo $t\in I$. Dado que $G_1$ y $a+b$ tienen signo constante, tienen el mismo signo.
\end{proof}
Los siguientes corolarios son aplicaciones directas de este resultado a los casos (C1')--(C3') respectivamente.
\begin{cor}\label{cor1sig} Bajo las hipótesis de $(C1^*)$ y el Teorema \ref{thmcases123}, $G_1$ tiene signo constante si
$$|A(T)|< \frac{\arccos(k)}{2\sqrt{1-k^2}}.$$
Es más, $\sign(G_1)=\sign(a)$.
\end{cor}
\begin{proof}
Las soluciones de \eqref{gen-eq} para el caso (C1'), como se vio antes, vienen dadas por
$$u(t)  =\a e^{-B_e(t)}\[ \cos\(\sqrt{1-k^2}A(t)\)-\frac{1+k}{\sqrt{1-k^2}}\sen\(\sqrt{1-k^2}A(t)\)\].$$
Usando un caso particular de la fórmula de la suma de fasores\footnote{$\a\cos \c+\b\sen \c=\sqrt{\a^2+\b^2}\sen(\c+\theta)$, donde $\theta\in[-\pi,\pi)$ es el ángulo tal que $\cos\theta=\frac{\b}{\sqrt{\a^2+\b^2}}$, $\sen\theta=\frac{\a}{\sqrt{\a^2+\b^2}}$.},
 $$u(t)=\a e^{-B_e(t)}\sqrt{\frac{2}{1-k}}\sen\(\sqrt{1-k^2}A(t)+\theta\),$$
 donde $\theta\in[-\pi,\pi)$ es el ángulo tal que
\begin{align}\label{sincos}\sen\theta=\sqrt{\frac{1-k}{2}}\quad\text{y}\quad\cos\theta=-\frac{1+k}{\sqrt{1-k^2}}\sqrt{\frac{1-k}{2}}=-\sqrt{\frac{1+k}{2}}.\end{align}
Obsérvese que esto implica $\theta\in\(\frac{\pi}{2},\pi\)$.
Para que las hipótesis de la Proposición \ref{proredpro} sean satisfechas, es necesario y suficiente pedir que $0\not\in u(I)$ para algún $\a\ne0$ o, equivalentemente, que 
$$\sqrt{1-k^2}A(t)+\theta\neq \pi n\sfa n\in\bZ\sfa t\in I,$$
Esto es,
$$A(t)\neq \frac{\pi n-\theta}{\sqrt{1-k^2}}\sfa n\in\bZ\sfa t\in I.$$
Puesto que $A$ es impar e inyectiva, $\theta\in\(\frac{\pi}{2},\pi\)$, esta última expresión es equivalente a
\begin{equation}\label{firststimate}|A(T)|< \frac{\pi-\theta}{\sqrt{1-k^2}}.\end{equation}
Ahora, usando la fórmula del ángulo doble para el seno y \eqref{sincos},
$$\frac{1-k}{2}=\sen^2\theta=\frac{1-\cos(2\theta)}{2}\text{,\quad esto es,\quad} k=\cos(2\theta),$$
lo cual implica, dado que $2\theta\in(\pi,2\pi)$,
$$\theta=\pi-\frac{\arccos(k)}{2},$$
donde $\arccos$ se define tal que su imagen es $[0,\pi)$. Sustituyendo esto en la desigualdad \eqref{firststimate} obtenemos
$$|A(T)|<\sigma(k):= \frac{\arccos(k)}{2\sqrt{1-k^2}},\quad k\in(-1,1).$$
\par
Usando que $|k|<1$, $a+b=(k+1)a+b_o$, la expresión explícita de la función de Green $G_1$ y la Proposición \ref{proredpro} podemos calcular el signo de $G_1$.
\end{proof}
\begin{rem}
En el caso en que $a=\omega$ es una constante y $k=0$, se tiene que $A(I)=[-|\omega|T,|\omega|T]$ y la condición del Corolario anterior puede ser reescrita como $|\omega|T<\frac{\pi}{4}$, lo cual es consistente con los resultados de \cite{Cab4}.\end{rem}
\begin{rem}
Obsérvese que $\sigma$ es estrictamente decreciente en $(-1,1)$ y que
$$\lim_{k\to-1^+}\sigma(k)=+\infty,\quad\lim_{k\to1^-}\sigma(k)=\frac{1}{2}.$$
\end{rem}
El siguiente Corolario estudia la situación equivalente en el caso (C3').
\begin{cor} Bajo las condiciones de (C3') y el Teorema \ref{thmcases123}, $G_{C3}$ tiene signo constante si $|A(T)|<\frac{1}{2}$.
\end{cor}
\begin{proof} Este Corolario es una consecuencia directa de \eqref{eqc3}, la Proposición \ref{proredpro} y el Teorema \ref{thmcases123}. Obsérvese que este resultado es consistente con $\sigma(1^-)=\frac{1}{2}$.
\end{proof}
Para demostrar el próximo resultado, necesitaremos la siguiente "versión hiperbólica" de la fórmula de la suma de fasores (ver \cite{Toj}).
\begin{lem}\label{hyppha}
Sea $\a$, $\b,\c\in\bR$, entonces
\begin{align*} & \a\cosh \c + \b\senh\c \\= & \sqrt{|\a^2-\b^2|}\begin{dcases}\cosh\(\frac{1}{2}\ln\left|\frac{\a+\b}{\a-\b}\right|+\c\) & \text{si}\quad \a>|\b|,
\\ -\cosh\(\frac{1}{2}\ln\left|\frac{\a+\b}{\a-\b}\right|+\c\) & \text{si}\quad- \a>|\b|,
\\
\senh\(\frac{1}{2}\ln\left|\frac{\a+\b}{\a-\b}\right|+\c\) & \text{si}\quad \b>|\a|,
\\
-\senh\(\frac{1}{2}\ln\left|\frac{\a+\b}{\a-\b}\right|+\c\) & \text{si }\quad -\b>|\a|,
\\
\a\,e^\c & \text{si }\quad\a=\b,\\
\a\,e^{-\c} & \text{si }\quad\a=-\b.\\
\end{dcases}\end{align*}
\end{lem}
\begin{cor}\label{cor2sig} Bajo las hipótesis de (C2') y el Teorema \ref{thmcases123}, $G_1$ tiene signo constante si $k<-1$ o
$$|A(T)|<-\frac{\ln(k-\sqrt{k^2-1})}{2\sqrt{k^2-1}}.$$
Además, $\sign(G_1)=\sign(k\,a)$.
\end{cor}
\begin{proof}
Las soluciones de \eqref{gen-eq} para el caso (C2'), como se vio antes, vienen dadas por
$$u(t)  =\a e^{-B_e(t)}\[\cosh\(\sqrt{k^2-1}A(t)\)- \frac{1+k}{\sqrt{k^2-1}}\senh\(\sqrt{k^2-1}A(t)\)\].$$
Si $k>1$, podemos deducir que $1<\frac{1+k}{\sqrt{k^2-1}}$ y, por lo tanto, usando el Lema \ref{hyppha},
 $$u(t)=-\a e^{-B_e(t)}\sqrt{\frac{2k}{k-1}}\senh\(\frac{1}{2}\ln\left|k-\sqrt{k^2-1}\right|+\sqrt{k^2-1}A(t)\),$$
Para que se satisfagan las hipótesis de la Proposición \ref{proredpro} es necesario y suficiente que $0\not\in u(I)$ para algún $\a\ne0$. Equivalentemente, que, 
$$\frac{1}{2}\ln(k-\sqrt{k^2-1})+\sqrt{k^2-1}A(t)\neq 0\sfa t\in I,$$
Esto es,
$$A(t)\neq -\frac{\ln(k-\sqrt{k^2-1})}{2\sqrt{k^2-1}}\sfa t\in I.$$
Puesto que $A$ es impar e inyectiva, esto es equivalente a
$$|A(T)|<\sigma(k):=-\frac{\ln(k-\sqrt{k^2-1})}{2\sqrt{k^2-1}},\quad k>1.$$
Ahora, si $k<-1$, tenemos que $\left|\frac{1+k}{\sqrt{k^2-1}}\right|<1$, por lo tanto, usando el Lema \ref{hyppha},
\begin{align*}u(t)=\a e^{-B_e(t)}\sqrt{\frac{2k}{k-1}}\cosh\(\frac{1}{2}\ln\left|k-\sqrt{k^2-1}\right|+\sqrt{k^2-1}A(t)\)\ne0 & \\ \text{para todo}\quad t\in I,\ \a\ne0 & ,\end{align*}
por lo que las hipótesis de la Proposición \ref{proredpro} se satisfacen.\par
Usando $|k|>1$, $a+b=(k^{-1}+1)b_e+b_o$, la expresión explícita de la función de Green $G_1$ y la Proposición \ref{proredpro} podemos calcular el signo de $G_1$.
\end{proof}
\begin{rem}
Si consideramos $\sigma$ definida a trozos como en los Corolarios \ref{cor1sig} y \ref{cor2sig} y la extendemos continuamente a través de $1/2$, obtenemos
$$\sigma(k):=\begin{cases} \frac{\arccos(k)}{2\sqrt{1-k^2}} & \text{ si } k\in(-1,1) \\
\frac{1}{2} & \text{ si } k=1 \\
-\frac{\ln(k-\sqrt{k^2-1})}{2\sqrt{k^2-1}} & \text{ si } k>1\end{cases}$$
Esta función no solo es continua (se ha definido como tal), sino también analítica. para poder mostrar esto basta con considerar la definición extendida del logaritmo y la raíz cuadrada a números complejos. Recordemos que $i:=\sqrt{-1}$ y que la rama principal del logaritmo se define como $\ln_0(z):=\ln|z|+i\theta$ donde $\theta\in[-\pi,\pi)$ y $z=|z|e^{i\theta}$ para todo $z\in\bC\backslash\{0\}$. Claramente, $\ln_0|_{(0,+\infty)}=\ln$.\par
Ahora, para $|k|<1$, $\ln_0(k-\sqrt{1-k^2}i)=i\theta$ con $\theta\in[-\pi,\pi)$ tal que $\cos\theta=k$, $\sen\theta=-\sqrt{1-k^2}$, esto es, $\theta\in[-\pi,0]$. Por lo tanto, $i\ln_0(k-\sqrt{1-k^2}i)=-\theta\in[0,\pi]$. Puesto que $\cos(-\theta)=k$, $\sen(-\theta)=\sqrt{1-k^2}$, resulta claro que
$$\arccos(k)=-\theta=i\ln_0(k-\sqrt{1-k^2}i).$$
Por lo tanto podemos extender $\arccos$ a $\bC$ como
$$\arccos(z):=i\ln_0(z-\sqrt{1-z^2}i),$$
lo cual es, evidentemente, una función analítica. Así, si $k>1$,
\begin{align*}  \sigma(k)= & -\frac{\ln(k-\sqrt{k^2-1})}{2\sqrt{k^2-1}}=-\frac{\ln_0(k-i\sqrt{1-k^2})}{2i\sqrt{1-k^2}}\\= & \frac{i\ln_0(k-i\sqrt{1-k^2})}{2\sqrt{1-k^2}}=\frac{\arccos(k)}{2\sqrt{1-k^2}}.\end{align*}
Además, $\sigma$ es positiva, estrictamente decreciente y 
$$\lim_{k\to -1^+}\sigma(k)=+\infty,\quad\lim_{k\to+\infty}\sigma(k)=0.$$
\end{rem}

\subt{Los casos (C4') y (C5')}\\
Los casos (C4') y (C5') son resonantes, es decir, no existe función de Green para ellos dado que las soluciones, de existir, son múltiples. Encontraremos la expresión de dichas soluciones a lo largo de esta Sección.\par Considérese ahora el siguiente problema derivado del problema \eqref{eq2cp}.

\begin{align}\label{eoparts}\begin{pmatrix}x_o' \\ x_e'\end{pmatrix} & =\begin{pmatrix}a_o-b_o & -a_e-b_e \\ a_e-b_e & -a_o-b_o\end{pmatrix}\begin{pmatrix}x_o \\ x_e\end{pmatrix}+\begin{pmatrix}h_e \\ h_o\end{pmatrix}.
\end{align}
Los siguientes teoremas nos muestran lo que ocurre cuando imponemos las condiciones de contorno.

\begin{thm} Si la condición (C4') se satisface, entonces el problema \eqref{eq2cp} tiene solución si y solo si
$$\int_0^Te^{B_e(s)}h_e(s)\dif s=0,$$
y en ese caso las soluciones de \eqref{eq2cp} vienen dadas por
\begin{align*}u_c(t)=e^{-B_e(t)}\[c+\int_0^t\(e^{B_e(s)}h(s)+2a_e(s)\int_0^se^{B_e(r)}h_e(r)\dif r\)\dif s\] & \\ \text{para } c\in\bR & .\end{align*}
\end{thm}
\begin{proof} Sabemos que cualquier solución del problema \eqref{eq2cp} tiene que satisfacer \eqref{eoparts}. En el caso (C4'), la matriz de \eqref{eoparts} es triangular inferior,
\begin{align}\label{eoparts3}\begin{pmatrix}x_o' \\ x_e'\end{pmatrix} & =\begin{pmatrix}-b_o & 0 \\ 2a_e & -b_o\end{pmatrix}\begin{pmatrix}x_o \\ x_e\end{pmatrix}+\begin{pmatrix}h_e \\ h_o\end{pmatrix}.
\end{align}
 y las soluciones de \eqref{eoparts3} vienen dadas por
 
\begin{align*}x_o(t) &
=e^{-B_e(t)}\[\theta+\int_0^te^{B_e(s)}h_e(s)\dif s\]
 ,\\x_e(t) & =e^{-B_e(t)}\[c+\int_0^t\(e^{B_e(s)}h_o(s)\right.\right.\\ &\left.\left.\quad+2a_e(s)\[\theta+\int_0^se^{B_e(r)}h_e(r)\dif r\]\)\dif s\],\end{align*}
donde $c$, $\theta\in\bR$. $x_e$ es par independientemente de $c$. Sin embargo, $x_o$ solo es impar cuando $\theta=0$. Por lo tanto, cualquier solución de \eqref{eq2cp}, si existe, ha de tener la forma
\begin{align*}u_c(t)=e^{-B_e(t)}\[c+\int_0^t\(e^{B_e(s)}h(s)+2a_e(s)\int_0^se^{B_e(r)}h_e(r)\dif r\)\dif s\] & \\ \text{para } c\in\bR. &\end{align*}
Para demostrar la segunda implicación es suficiente con comprobar que $u_c$ es una solución del problema \eqref{eq2cp}.
\begin{align*}
& u'_c(t) \\ = & -b_o(t)e^{-B_e(t)}\[c+\int_0^t\(e^{B_e(s)}h(s)+2a_e(s)\int_0^se^{B_e(r)}h_e(r)\dif r\)\dif s\] \\& \quad+e^{-B_e(t)}\(e^{B_e(t)}h(t)+2a_e(t)\int_0^te^{B_e(r)}h_e(r)\dif r\)\\= & h(t)-b_o(t)u_c(t)+2a_e(t)e^{-B_e(t)}\int_0^te^{B_e(r)}h_e(r)\dif r.
\end{align*}
Ahora bien,
\begin{align*}
& a_e(t)(u_c(-t)-u_c(t))+2a_e(t)e^{-B_e(t)}\int_0^te^{B_e(r)}h_e(r)\dif r\\= & a_e(t)e^{-B_e(t)}\[c-\int_0^t\(e^{B_e(s)}h(-s)-2a_e(s)\int_0^se^{B_e(r)}h_e(r)\dif r\)\dif s\]\\ &\quad -  a_e(t)e^{-B_e(t)}\[c+\int_0^t\(e^{B_e(s)}h(s)+2a_e(s)\int_0^se^{B_e(r)}h_e(r)\dif r\)\dif s\]\\ &\quad+2a_e(t)e^{-B_e(t)}\int_0^te^{B_e(r)}h_e(r)\dif r\\= & -2a_e(t)e^{-B_e(t)}\int_0^te^{B_e(r)}h_e(r)\dif s+2a_e(t)e^{-B_e(t)}\int_0^te^{B_e(r)}h_e(r)\dif r=  0.
\end{align*}
Por tanto,
$$u'(t)+a_e(t)u(-t)+(-a_e(t)+b_o(t))u(t)=h(t).$$
La condición de contorno $u(-T)-u(T)=0$ es equivalente a $u_o(T)=0$, esto es,
$$\int_0^Te^{B_e(s)}h_e(s)\dif s=0$$
y por lo tanto la demostración queda terminada.
\end{proof}

\begin{thm} Si la condición (C5') se satisface, entonces el problema \eqref{eq2cp} tiene solución si y solo si
\begin{equation}
\label{conodd2}
\int_0^T e^{B(s)-A(s)}h_e(s)\dif s=0,\end{equation}
y en este caso las soluciones de \eqref{eq2cp} vienen dadas por 
$$u_c(t)=e^{A(t)}\int_0^te^{-A(s)}h_e(s)\dif s+e^{-A(t)}\[c+\int_0^te^{A(s)}h_o(s)\dif s\]\nkp\text{con } c\in\bR.$$
\end{thm}
\begin{proof}  En el caso (C5'), $b_o=b$ y $a_o=a$. Además, la matriz de \eqref{eoparts} es diagonal,
\begin{align}\label{eoparts5}\begin{pmatrix}x_o' \\ x_e'\end{pmatrix} & =\begin{pmatrix}a_o-b_o & 0 \\ 0 & -a_o-b_o\end{pmatrix}\begin{pmatrix}x_o \\ x_e\end{pmatrix}+\begin{pmatrix}h_e \\ h_o\end{pmatrix}.
\end{align}
 y las soluciones de \eqref{eoparts5} vienen dadas por

\begin{align*}x_o(t) & =e^{A(t)-B(t)}\[\theta+\int_0^te^{B(s)-A(s)}h_e(s)\dif s\],\\x_e(t) & =e^{-A(t)-B(t)}\[c+\int_0^te^{A(s)+B(s)}h_o(s)\dif s\],\end{align*}
donde $c$, $\theta\in\bR$. Puesto que $a$ es impar, $A$ es par, luego $x_e$ es par independientemente del valor de $c$. Sin embargo, $x_o$ es impar solo cuando $\theta=0$. En tal caso, puesto que necesitamos $x_o(T)=0$, obtenemos la condición \eqref{conodd2}, que nos permite obtener la primera implicación del Teorema.\par
Cualquier solución $u$ de \eqref{eq2cp} ha de tener la siguiente expresión
\begin{align*} & u_c(t)=\\ & e^{A(t)-B(t)}\int_0^te^{B(s)-A(s)}h_e(s)\dif s+e^{-A(t)-B(t)}\[c+\int_0^te^{A(s)+B(s)}h_o(s)\dif s\].\end{align*}
Para demostrar la segunda implicación es suficiente con demostrar que $u_c$ es una solución del problema \eqref{eq2cp}.
\begin{align*}
& u_c'(t)\\ =&(a(t)-b(t))e^{A(t)-B(t)}\int_0^te^{B(s)-A(s)}h_e(s)\dif s\\ &\quad-(a(t)+b(t))e^{-A(t)-B(t)}\[c+\int_0^te^{A(s)+B(s)}h_o(s)\dif s\]+h(t).\end{align*}
Ahora bien,
\begin{align*}\ & a(t)u_c(-t)+b(t)u_c(-t)\\=& a(t)\(-e^{A(t)-B(t)}\int_0^te^{B(s)-A(s)}h_e(s)\dif s\right.\\ &\left.\quad+e^{-A(t)-B(t)}\[c+\int_0^te^{A(s)+B(s)}h_o(s)\dif s\]\)\\ & \quad +b(t)\(e^{A(t)-B(t)}\int_0^te^{B(s)-A(s)}h_e(s)\dif s\right.\\ &\left.\quad+e^{-A(t)-B(t)}\[c+\int_0^te^{A(s)+B(s)}h_o(s)\dif s\]\)\\ = &-(a(t)-b(t))e^{A(t)-B(t)}\int_0^te^{B(s)-A(s)}h_e(s)\dif s\\ &\quad+(a(t)+b(t))e^{-A(t)-B(t)}\[c+\int_0^te^{A(s)+B(s)}h_o(s)\dif s\].
\end{align*}
Por lo que, claramente,
$$u_c'(t)+a(t)u_c(-t)+b(t)u_c(t)=0\quad\text{\ct } t\in I.$$
lo cual termina la demostración.
\end{proof}

\subt{Casos mixtos}\\
Cuando no nos encontramos en los casos (C1')--(C5'), es más difícil precisar cuando el problema \eqref{eq2cp} tiene solución. Aquí presentamos algunos resultados parciales.
\par
Considérese la siguiente ecuación diferencial con condiciones de contorno periódicas,
\begin{equation}\label{usualp}x'(t)+[a(t)+b(t)]x(t)=0,\quad x(-T)=x(T)\end{equation}
El siguiente Lema nos da la expresión de la función de Green para dicho problema. Definimos $\upsilon=a+b$.
\begin{lem}\label{lem1} Sean $h$, $a$ en el problema \eqref{usualp} funciones de $L^1(I)$ y supongamos que $\int_{-T}^{T}\upsilon(t)\dif t\ne0$. Entonces el problema \eqref{usualp} tiene solución única dada por
$$u(t)=\int_{-T}^{T}G_3(t,s)h(t)\dif s,$$
donde
$$G_3(t,s)=\begin{cases}\tau\,e^{\int_t^s\upsilon(r)\dif r}, & s\le t,\\(\tau-1)e^{\int_t^s\upsilon(r)\dif r}, & s> t,\end{cases}\quad \text{y}\quad \tau=\frac{1}{1-e^{-\int_{-T}^T\upsilon(r)\dif r}}.$$
\end{lem}
\begin{proof}
$$\frac{\partial G_3}{\partial t}(t,s)=\begin{cases}-\tau\,\upsilon(t)\,e^{\int_t^s\upsilon(r)\dif r}, & s\le t,\\-(\tau-1)\upsilon(t)e^{\int_t^s\upsilon(r)\dif r}, & s> t,\end{cases}=-\upsilon(t)G_3(t,s).$$
Por lo tanto,
$$\frac{\partial G_3}{\partial t}(t,s)+\upsilon(t)G_3(t,s)=0,\ s\ne t.$$
En consecuencia,
\begin{align*}
& u'(t)+\upsilon(t)u(t)=\frac{\dif}{\dif t}\int_{-T}^{t}G_3(t,s)h(t)\dif s+\frac{\dif}{\dif t}\int_{t}^{T}G_3(t,s)h(t)\dif s\\ &\quad+\upsilon(t)\int_{-T}^{T}G_3(t,s)h(t)\dif s\\= & [G_3(t,t^-)-G_3(t,t^+)]h(t)\\ &\quad+\int_{-T}^{T}\[\frac{\partial G_3}{\partial t}(t,s)+\upsilon(t)G_3(t,s)\]h(t)\dif s\\ = & h(t)\nkp\text{\ct\ }t\in I.
\end{align*}
Las condiciones de contorno también se satisfacen.
\begin{align*}
& u(T)-u(-T)=\int_{-T}^{T}\[\tau\,e^{\int_T^s\upsilon(r)\dif r}-(\tau-1)e^{\int_{-T}^s\upsilon(r)\dif r}\]h(s)\dif s\\= & \int_{-T}^{T}
\[\frac{e^{\int_T^s\upsilon(r)\dif r}}{1-e^{-\int_{-T}^T\upsilon(r)\dif r}}-\frac{e^{-\int_{-T}^T\upsilon(r)\dif r}\,e^{\int_{-T}^s\upsilon(r)\dif r}}{1-e^{-\int_{-T}^T\upsilon(r)\dif r}}\]h(s)\dif s\\= & \int_{-T}^{T}
\[\frac{e^{\int_T^s\upsilon(r)\dif r}}{1-e^{-\int_{-T}^T\upsilon(r)\dif r}}-\frac{e^{-\int_{T}^s\upsilon(r)\dif r}}{1-e^{-\int_{-T}^T\upsilon(r)\dif r}}\]h(s)\dif s=0.
\end{align*}
\end{proof}
\begin{lem} $$|G_3(t,s)|\le F(\upsilon):=\frac{e^{\|\upsilon\|_1}}{|e^{\|\upsilon^+\|_1}-e^{\|\upsilon^-\|_1}|}.$$
\end{lem}
\begin{proof}
Obsérvese que
$$\tau=\frac{1}{1-e^{\|\upsilon^-\|_1-\|\upsilon^+\|_1}}=\frac{e^{\|\upsilon^+\|_1}}{e^{\|\upsilon^+\|_1}-e^{\|\upsilon^-\|_1}}.$$
Luego
$$\tau-1=\frac{e^{\|\upsilon^-\|_1}}{e^{\|\upsilon^+\|_1}-e^{\|\upsilon^-\|_1}}.$$
Por otra parte,
$$e^{\int_t^s\upsilon(r)\dif r}\le\begin{cases} e^{\|\upsilon^-\|_1}, & s\le t,\\e^{\|\upsilon^+\|_1}, & s>t,
\end{cases}$$
lo cual termina la demostración.
\end{proof}
El siguiente resultado demuestra la existencia y unicidad de solución del problema $\eqref{eq2cp}$ cuando $\upsilon$ es "suficientemente pequeña".
\begin{thm}\label{thmpn}Sean $h$, $a$, $b$ en el problema \eqref{eq2cp} funciones de  $L^1(I)$ y supongamos que $\int_{-T}^{T}\upsilon(t)\dif t\ne0$. Sea $$M:=\{(2T)^\frac{1}{p}(\|a\|_{p^*}+\|b\|_{p^*})\}_{p\in[1,+\infty]}$$ donde $p^{-1}+(p^*)^{-1}=1$. Si $F(\upsilon)\|a\|_1(\inf M)<1$, entonces el problema \eqref{eq2cp} tiene solución única.
\end{thm}
\begin{proof}
Manipulando los términos adecuadamente obtenemos
 \begin{eqnarray*}
  h(t) & = & x'(t)+a(t)\(\int_t^{-t}x'(s)\dif s+x(t)\)+b(t)x(t) \\  & = & x'(t)+\upsilon(t)x(t)+a(t)\int_t^{-t}(h(s)-a(s)x(-s)-b(s)x(s))\dif s.
  \end{eqnarray*}
Así,
\begin{align*} & u'(t)+\upsilon(t)x(t) \\ = & a(t)\int_t^{-t}(a(s)x(-s)+b(s)x(s))\dif s+a(t)\int_{-t}^{t}h(s)\dif s+h(t).\end{align*}
Usando el Lema \ref{lem1} resulta evidente que
\begin{align*}x(t) = &\int_{-T}^TG_3(t,s)a(s)\int_s^{-s}(a(r)x(-r)+b(r)x(r))\dif r\dif s\\ &\quad+\int_{-T}^TG_3(t,s)\[a(s)\int_{-s}^{s}h(r)\dif r+h(s)\]\dif s\end{align*}
esto es, $x$ es un punto fijo de un operador de la forma $Hx(t)+\beta(t)$, por lo tanto, por el Teorema de contracción de Banach, es suficiente con demostrar que $\|H\|<1$ para alguna norma compatible.
\par
Usando el Teorema de Fubini,
$$Hx(t)=-\int_{-T}^{T}\rho(t,r)(a(r)x(-r)+b(r)x(r))\dif r,$$
donde $\rho(t,r)=\[\int_{|r|}^{T}-\int_{-T}^{-|r|}\]G_3(t,s)a(s)\dif s$.

Si $\int_{-T}^{T}\upsilon(t)\dif t=\|\upsilon^+\|_1-\|\upsilon^-\|_1>0$ entonces $G_3$ es positiva y
$$\rho(t,r)\le\int_{-T}^{T}G_3(t,s)|a(s)|\dif s\le F(\upsilon)\|a\|_1.$$
Tenemos la misma acotación para $-\rho(t,r)$.
Si $\int_{-T}^{T}\upsilon(t)\dif t<0$ procedemos con un argumento análogo y concluimos también que
$|\rho(t,s)|<F(\upsilon)\|a\|_1$. Así, 
\begin{align*}|Hx(t)| \le & F(\upsilon)\|a\|_1\int_{-T}^{T}|a(r)x(-r)+b(r)x(r)|\dif r \\ = & F(\upsilon)\|a\|_1\|a(r)x(-r)+b(r)x(r)\|_1.\end{align*}
Por tanto, se tiene que
$$\|Hx\|_p  \le (2T)^\frac{1}{p}F(\upsilon)\|a\|_1(\|a\|_{p^*}+\|b\|_{p^*})\|x\|_p,\ p\in[1,\infty],$$
con lo que finaliza la demostración.
\end{proof}
\begin{rem}Bajo las hipótesis del Teorema \ref{thmpn}, obsérvese que $F(\upsilon)\ge 1$.
\end{rem}
El siguiente resultado nos permitirá obtener alguna información acerca del signo de la solución del problema \eqref{eq2cp}. Para demostrarlo, usaremos un Teorema de \cite{Cab5} que citamos abajo.
\par
Considérese un intervalo $[w,d]\subset I$, el cono
\begin{equation*}\label{eqcone-cs}
K=\{u\in \cC(I): \min_{t \in [w,d]}u(t)\geq c \|u\|\},
\end{equation*}
y el problema
\begin{equation}\label{eqgenpro2}
x'(t)  =h(t,x(t),x(-t)),\,  t\in I,\quad x(-T)=x(T).
\end{equation}

Considérense las siguientes condiciones.

\begin{enumerate}
\item[$(\mathrm{I}_{\protect\rho,\omega}^{1})$] \label{EqB2} Existen $\rho> 0$ y $\omega\in\(0,\frac{\pi}{4T}\]$ tales que $f^{-\rho,\rho}_\omega <\omega$ donde
\begin{align*}
 &   f^{{-\rho},{\rho}}_\omega:= \\ & \sup \left\{\frac{h(t,u,v)+\omega v}{\rho }:\;(t,u,v)\in
[ -T,T]\times [ -\rho,\rho ]\times [-\rho,\rho ]\right\}.\end{align*}

\item[$(\mathrm{I}_{\protect\rho,\omega}^{0})$] Existe $\rho >0$ tal que
$$
f_{(\rho ,{\rho /c})}^\omega\cdot\inf_{t\in [w,d]}\int_{w}^{d}\ol G(t,s)\,ds>1,
$$
donde
\begin{align*}
& f_{(\rho ,{\rho /c})}^\omega =\\ & \inf \left\{\frac{h(t,u,v)+\omega v}{\rho }%
:\;(t,u,v)\in [w,d]\times [\rho ,\rho /c]\times
[-\rho /c,\rho /c]\right\}.\end{align*}
\end{enumerate}

\begin{thm}\textrm{\cite[Teorema 5.15]{Cab5}}\label{thmgen}

Sea $\omega\in\(0,\frac{\pi}{2}T\]$. Sea $[w,d]\subset I$ tal que $$w=T-d\in\(\max\left\{0,T-\frac{\pi}{4\omega}\right\},\frac{T}{2}\).$$ Definamos
$$c=\frac{[1-\tan(\omega d)][1-\tan(\omega w)]}{[1+\tan(\omega d)][1+\tan(\omega w)]}.$$

Entonces el problema \eqref{eqgenpro2} tiene al menos una solución no trivial en $K$ si alguna de las siguientes condiciones se satisface.

\begin{enumerate}

\item[$(S_{1})$] Existen $\rho _{1},\rho _{2}\in (0,\infty )$ con $\rho
_{1}/c<\rho _{2}$ tales que $(\mathrm{I}_{\rho _{1},\omega}^{0})$ y $(\mathrm{I}_{\rho _{2},\omega}^{1})$ se satisfacen.

\item[$(S_{2})$] Existen $\rho _{1},\rho _{2}\in (0,\infty )$ con $\rho
_{1}<\rho _{2}$ tales que $(\mathrm{I}_{\rho _{1},\omega}^{1})$ y $(\mathrm{I}%
_{\rho _{2},\omega}^{0})$ se satisfacen.
\end{enumerate}

\end{thm}

\begin{thm}\label{thmmix2}
Sean $h\in L^\infty(I)$, $a,b\in L^1(I)$ tal que $0<|b|<a<\omega<\frac{\pi}{2}T$ \ctp\ e $\inf h>0$. Entonces existe una solución $u$ del problema \eqref{eq2cp} tal que, si $u\not\equiv0$, $u>0$ en el intervalo $\(\max\{0,T-\frac{\pi}{4\omega}\},\min\{T,\frac{\pi}{4\omega}\}\)$.
\end{thm}
\begin{proof}
El problema \eqref{eq2cp} puede ser rescrito como
\begin{equation*}\label{eq2cp3} x'(t)=h(t)-b(t)\,x(t)-a(t)\,x(-t),\nkp t\in I,\quad x(-T)=x(T).
\end{equation*}
Con esta formulación, podemos aplicar el Teorema \ref{thmgen}.
Dado que $0<a-|b|<\omega$ \ctp, tomamos $\rho_2\in\bR^+$ tal que $h<(a-|b|)\rho_2$ \ctp\ En consecuencia, $h<(a-\omega)\rho_2-|b|\rho_2+\rho_2\omega$ \ct\ $t\in I$, en particular, $$h<(a-\omega)v-|b|u+\rho_2\omega\le(a-\omega)v+b\,u+\rho_2\omega$$
para casi todo $t\in I;\ u,v\in[-\rho_2,\rho_2].$ Como consecuencia,
$$\sup \left\{\frac{h(t)-b(t)u-a(t)v+\omega v}{\rho_2 }:\;(t,v)\in
[ -T,T]\times [-\rho_2,\rho_2]\right\}<\omega,$$
y así, $(\mathrm{I}_{\rho _{2},\omega}^{1})$ se satisface.\par
Sea $[w,d]\subset I$ tal que $[w,d]\subset\(T-\frac{\pi}{4\omega},\frac{\pi}{4\omega}\)$. Sea $$c=\frac{[1-\tan(\omega d)][1-\tan(\omega w)]}{[1+\tan(\omega d)][1+\tan(\omega w)]},$$
y $\e=\omega\int_w^d\ol G(t,s)\dif s$. Elijamos $\d\in(0,1)$ tal que $$h>\[\(1+\frac{c}{\e}\)\omega-(a-|b|)\]\rho_2\d$$  \ctp y definamos $\rho_1:=\d c\rho_2$. En consecuencia, $h>\[(a-\omega)v+b\,u\]\frac{\omega}{\e}\rho_1$ \ct\ $t\in I$, $u\in[\rho_1,\frac{\rho_1}{c}]$ y $v\in[-\frac{\rho_1}{c},\frac{\rho_1}{c}]$. Por tanto,
$$\inf \left\{\frac{h(t)-b(t)u-a(t)v+\omega v}{\rho_1 }:\;(t,v)\in [w,d]\times
[-\rho_1/c,\rho_1/c]\right\}>\frac{\omega}{\e},$$
y entonces  $(\mathrm{I}_{\rho _{1},\omega}^{0})$ se satisface.
Finalmente podemos afirmar que $(S_1)$ en el Teorema \ref{thmgen} se satisface y obtenemos el resultado deseado.
\end{proof}
\begin{rem} En las hipótesis del Teorema \ref{thmmix2},  if $\omega<\frac{\pi}{4}T$, podemos tomar $[w,d]=[-T,T]$ y continuar con la demostración del Teorema \ref{thmmix2} como hicimos arriba. Esto garantiza que la solución $u$ es positiva.
\end{rem}

\end{document}